\documentclass[a4paper]{elsarticle}
\usepackage{lipsum}
\makeatletter
\def\ps@pprintTitle{%
 \let\@oddhead\@empty
 \let\@evenhead\@empty
 \def\@oddfoot{}%
 \let\@evenfoot\@oddfoot}
\makeatother
\usepackage[left=3cm,right=3cm,top=2cm,bottom=2cm,includeheadfoot]{geometry}
\usepackage{amsmath}
\usepackage{amsfonts}
\usepackage{amssymb}
\usepackage{amsfonts}
\usepackage{ MnSymbol}
\usepackage{amsmath}
\usepackage{cancel}
\usepackage{graphicx}
\usepackage[numbers]{natbib}

\usepackage[ansinew]{inputenc}
\usepackage{amsthm}
\usepackage[arrow, matrix, curve]{xy}
\usepackage{hyperref}
\newtheoremstyle{style}   
 {0.5cm}                 %Space above    
 {0.5cm}                 %Space below    
  {}                         %Body font: original {\normalfont}    
  {}                         %Indent amount (empty = no indent,%\parindent = paraindent)    
  {\normalfont\bfseries}  %Thm head font original 
  {\normalfont }{   }
  {}
\newtheorem{thm}{Theorem}
\newtheorem{satz}[thm]{Theorem}
\newtheorem{Proposition}[thm]{Proposition}

\theoremstyle{remark}
\newtheorem{remark}{Remark}
\newcommand{\ph}{\varphi}
\newcommand{\R}{\mathbb{R}}

\newcommand{\C}{\mathbb{C}}

\newcommand{\De}{\Delta_{p,q}}

\newcommand{\Pe}{\mathcal{P}}
\newcommand{\Q}{\mathcal{Q}}

\allowdisplaybreaks[1]
\begin{document}

\begin{frontmatter}

\title{{Charged Conformal Killing Spinors}}

%\author[label1]{A. Lischewski}
%\address[label1]{Humboldt Uni Berlin}

\author{Andree Lischewski}

\address{Department of Mathematics, Humboldt University, Rudower Chausse 25, 12489 Berlin, Germany}
\date{today}
\begin{abstract}
We study the twistor equation on pseudo-Riemannian $Spin^c-$manifolds whose solutions we call charged conformal Killing spinors (CCKS). We derive several integrability conditions for the existence of CCKS and study their relations to spinor bilinears. A construction principle for Lorentzian manifolds admitting CCKS with nontrivial charge starting from CR-geometry is presented. We obtain a partial classification result in the Lorentzian case under the additional assumption that the associated Dirac current is normal conformal and complete the Classification of manifolds admitting CCKS in all dimensions and signatures $\leq 5$ which has recently been initiated in the study of supersymmetric field theories on curved space.
\end{abstract}

\begin{keyword}
twistor spinors \sep $Spin^c$-Geometry \sep conformal Killing forms 
\MSC[2010] 53B30 \sep 53A30 \sep 53C27 \sep 15A66  
%(2000 is the default)

\end{keyword}
\ead{lischews@math.hu-berlin.de}

\end{frontmatter}
\tableofcontents
\section{Introduction}
The study of pseudo-Riemannian geometries admitting symmetries or conformal symmetries is a classical problem in differential geometry, cf. \cite{kr1,kr2,kr3,dude}. The spinorial analogue leads to the determination of manifolds on which certain spinor field equations can be solved. The pseudo-Riemannian Berger list opens up a way to distinguish the holonomy groups of irreducible geometries admitting parallel spinors, see \cite{bk}. Furthermore, Lorentzian manifolds with special holonomy admitting parallel spinors or pseudo-Riemannian geometries with parallel pure spinor fields have been studied intensively in \cite{ldr,bll,bv6,kath}. A list of local normal forms of the metric is known in low dimension, see \cite{br}.\\
Generalizing the concept of a parallel spinor, the spinorial analogue of Killing vector fields leads to (geometric) Killing spinors which -at least in the Riemannian and Lorentzian case- haven been well-studied in \cite{bfkg,boh,kath,leiks} and many construction principles are known. Interest in these objects arose independently from the fact that as shown in \cite{fr1} on a a compact Riemannian spin manifold the eigenspinors to the minimal possible eigenvalue of the Dirac operator are Killing spinors. Moreover, \cite{baer} relates Killing spinors to parallel spinors on the cone. It is natural to consider a generalization of this problem to conformal geometry giving rise to the study of conformal Killing spinors, or twistor spinors. They lie in the kernel of a natural differential operator acting on spinor bundles which can be interpreted as being complementary to the spin Dirac operator. Local geometries admitting twistor spinors have been intensively studied in \cite{bfkg,bl,lei,leihabil,kr4,kr5,ha96} for the Riemannian and Lorentzian case. However, also the study of the twistor equation in higher signatures is of interest as indicated in \cite{hs1,hs,lis}. Among other aspects it leads to a spinorial characterization of 5-manifolds admitting generic 2-distributions and to new construction principles for projective structures. Twistor spinors square to conformal vector fields with the special additional property that they insert trivially into the Weyl-and Cotton tensor, see \cite{bl,raj,med} for which the term normal conformal vector field has become standard in the literature. A generalization of this property to differential forms has been studied in \cite{nc,sem}, leading to new classification results for pseudo-Riemannian decomposable conformal holonomy, cf. \cite{baju,al}.\\
The study of these spinor field equations has also been motivated by progress in the understanding of physical theories with supergravity and vice versa. For instance, Riemannian manifolds admitting parallel or Killing spinors allow one to place certain supersymmetric Yang Mills theories on them, see \cite{jf5,ym}. In physics, the twistor equation first appeared in \cite{pen}. Moreover, the generalized Killing spinor equations appearing in the Freund-Rubin product ansatz for 11-dimensional supergravity (cf. \cite{of12}) lead to conformal Killing spinor equations on the factors. Recently, one has started to place and study some supersymmetric Minkowski-space theories on curved space which may lead to new insights in the computation of observables, see \cite{pes,ym,fes1,CCKS1,CCKS2,CCKS3}. Requiring that the deformed theory on curved space preserves some (rigid) supersymmetry (cf. \cite{fes1}) again leads to generalized Killing spinor equations. Interestingly, one finds in \cite{CCKS1,CCKS2,CCKS3} for different theories and signatures, namely Euclidean and Lorentzian 3-and 4 manifolds the same type of spinorial equation, namely a $Spin^c$-analogue of the twistor spinor equation whose solutions have been named charged conformal Killing spinors (CCKS). As shown in these references, one obtains this twistor equation also by using the AdS/CFT-correspondence and studying the gravitino-variation near the conformal boundary.\\
\newline
In order to put these local results into a more global mathematical context, consider a space- and time-oriented, connected pseudo-Riemannian $Spin^c$ manifold $(M,g)$ of signature $(p,q)$ with underlying $S^1$-principal bundle $\Pe_1$. One can canonically associate to this setting  the complex spinor bundle $S^g$ with its Clifford multiplication, denoted by $\mu : TM \times S^g \rightarrow S^g$. If moreover a connection $A$ on $\Pe_1$ is given, there is a canonically induced covariant derivative $\nabla^{A}$ on $S^g$. Besides the Dirac operator $D^A$, there is another conformally covariant differential operator acting on spinor fields, obtained by performing the spinor covariant derivative $\nabla^{A}$ followed by orthogonal projection onto the kernel of Clifford multiplication,
\[ P^A : \Gamma(S^g)  \stackrel{\nabla^{A}}{\rightarrow} \Gamma(T^*M \otimes S^g ) \stackrel{g}{\cong} \Gamma(TM \otimes S^g)  \stackrel{\text{proj}_{\text{ker}\mu}}{\rightarrow} \Gamma(\text{ker }\mu), \]
called the $Spin^c$-twistor operator. Elements of its kernel are precisely CCKSs and they are equivalently characterized as solutions of the conformally covariant $Spin^c$-twistor equation
\begin{align}
\nabla^{A}_X \ph + \frac{1}{n} X \cdot D^A \ph = 0 \text{    for all } X \in \mathfrak{X}(M). \label{tatar}
\end{align}
This article is devoted to the study of the twistor equation on $Spin^c$-manifolds. From a physics perspective, the motivation for its study lies in the determination of geometries in dimensions 3 and 4 on which supersymmetric field theories can be placed. In these signatures, (\ref{tatar}) has been solved locally in \cite{CCKS1,CCKS2,CCKS3}. However, there are compelling additional purely geometric reasons for the study of (\ref{tatar}). First, it is a natural generalization of $Spin^c$-parallel and Killing spinors which have been investigated in \cite{mor}. Their study has lead to new spinorial characterizations of Sasakian and pseudo-K\"ahler structures. Generalizations of the $Spin^c-$Killing spinor equations have appeared in \cite{gross}. Moreover, CCKS might lead to equivalent characterizations of manifolds admitting certain conformal Killing forms: Given a CCKS $\ph$, one can always form its associated Dirac current $V_{\ph}$. In the $Spin-$case, i.e. $dA=0$, $V_{\ph}$ is always a causal, normal conformal vector field. However, for Lorentzian 3-manifolds it has been shown in \cite{CCKS2} that for every non spacelike conformal vector field $V$ there is a CCKS $\ph$ wrt. a generically non-flat connection $A$ such that $V=V_{\ph}$. The same holds on Lorentzian 4-manifolds for lightlike conformal vector fields: \cite{lei} shows that for $(M^{3,1},g)$ a non-conformally flat Lorentzian manifold admitting a null normal conformal vector field $V$ without zeroes such that its twist $V^{\flat} \wedge dV^{\flat}$ vanishes everywhere or nowhere on $M$, there exists locally a real twistor spinor $\ph \in \Gamma(S^g)$ such that $V_{\ph} = V$. 
In view of this, we ask whether the existence of a generic null conformal vector field on $(M^{3,1},g)$ which is not necessarily normal conformal can be characterized in terms of spinor fields, which has been answered in  \cite{CCKS1}. Namely for $V$ a null conformal vector field without zeroes on a Lorentzian manifold $(M^{3,1},g)$ there exists locally a connection $A$ and a CCKS $\ph \in \Gamma(S^g)$ wrt. $A$ such that $V=V_{\ph}$.
We want to investigate whether this principle carries over also to other signatures. This would lead to spinorial characterizations of manifold admitting certain conformal symmetries via CCKS.
Consequently, the study of supersymmetric theories on curved space as well as the natural generalization of already studied $Spin^c$ spinor field equations together with the question of what the spinorial analogue of conformal, not necessarily normal conformal vector fields might be in low dimensions motivates the study of the twistor equation on pseudo-Riemannian $Spin^c$-manifolds in full generality.\\
\newline
This article starts with the investigation of basic properties of the $Spin^c$-twistor operator. In particular, its conformal covariance is revealed where the $S^1-$bundle data remain unchanged. It is straightforward to derive integrability conditions in section \ref{ikon} relating the conformal Weyl curvature tensor $W^g$ to the curvature $dA$ of the $S^1$-connection. Moreover, we study forms and vector fields associated to a CCKS via a natural squaring map as known from \cite{bl,nc,cortes}. They turn out to provide examples for conformal, not necessarily normal conformal Killing forms. In section \ref{crg} we then ask for global construction principles of Lorentzian manifolds admitting solutions of the CCKS equation. Here, one is motivated by some known constructions: Every pseudo-Riemannian Ricci-flat K\"ahler spin manifold admits (at least) 2 parallel spinors, see \cite{bk}. Given a K\"ahler manifold equipped with its canonical $Spin^c$-structure and the $S^1$-connection $A$ canonically induced by the Levi-Civita connection, \cite{mor} shows that there is (generically) one $Spin^c$-parallel spinor wrt. $A$ and $dA=0$ iff the manifold is Ricci flat. Passing to conformal geometry, it is known that Fefferman spin spaces over strictly pseudoconvex manifolds (cf. \cite{bafe,baju,leihabil}) can be viewed as the Lorentzian and conformal analogue of Riemannian K\"ahler manifolds and that they always admit 2 conformal Killing spinors. Further properties of these distinguished twistor spinors can even be used to characterize Fefferman spin spaces. This construction is presented in detail in \cite{bafe} and from a conformal holonomy point of view in \cite{baju,leihabil}. In view of this, it is natural to conjecture that there is a conformal $Spin^c$-analogue. Indeed, we find in Theorem \ref{feff} that every Fefferman space $(F^{2n+2},h_{\theta})$ over a strictly pseudoconvex manifold $(M^{2n+1},H,J,\theta)$ admits a canonical $Spin^c$-structure and a natural $S^1$-connection $A$ on the auxiliary bundle induced by the Tanaka Webster connection on $M$ such that there exists a CCKS on $F$. Under additional natural assumptions we show in Theorem \ref{converse} that also the converse direction is true. Thus, one obtains a characterization of Fefferman spaces in terms of $Spin^c$-spinor equations. It seems very natural to characterize Fefferman spaces in terms of distinguished $Spin^c$-spinor fields as every Fefferman space over a strictly pseudoconvex CR-manifold admits a natural $Spin^c$-structure. For the classical characterization from \cite{bafe,baju} in terms of ordinary twistor spinors, one has to restrict to the class of Fefferman spin-spaces.\\
Further, we obtain in section \ref{pcr} a classification of local Lorentzian geometries admitting CCKS under the additional assumption that the associated conformal vector field is normal conformal in Theorem \ref{clar}. Hereby, we use known results about conformal holonomy for Lorentzian geometries as known from \cite{leihabil,lcc,al,baju}. \\
Finally, we complete in section \ref{lodi} the classification of local geometries admitting solutions to (\ref{tatar}) which has been initiated in \cite{CCKS1,CCKS2,CCKS3}. Hereby, our study of the $Spin^c-$twistor equation on Lorentzian 5-manifolds leads to an equivalent spinorial characterization of geometries admitting Killing 2-forms of a certain causal type in Theorem \ref{satz3}. It is straightforward to obtain similar results in signature $(3,2)$. For signatures $(2,2)$ we find in Theorem \ref{22} that CCKS of nonzero length equivalently characterize the local existence of pseudo-K{\"a}hler metrics in the conformal class. Considering an irreducible Riemannian 5-manifold, we deduce in Theorem \ref{50} that locally there exists a CCKS if and only if there is a Sasakian (not necessarily Einstein-Sasakian) structure in the conformal class. %Consequently, in low dimensions CCKS provide a spinorial characterization of the existence of certain standard geometric structures in the conformal class.

%This article is organized as follows: In section 2 we introduce the basic ingredients of conformal $Spin^c-$geometry in arbitrary signature and show how CCKS can be described as parallel sections in the double spinor bundle wrt. a suitbale connection. Sections 3 investigates the integrability conditions resulting from the CCKS equation, the relations between the Weyl curvature and the curvature of the $S^1-$connection and the properties of the spinor bilinears constructed of nof \cite{bafe}. Based on the results obtained so far, we can then present a partial classification result in section $5$. In section 6 we continue the local analysis of the CCKS equation which has been initiated recently in physics literature and end up with a local geometric description of geometries admitting CCKS in signatures $(0,5),(2,2)$ and $(3,2)$.

\section{$Spin^c$-Geometry and the twistor operator} \label{srty}
\subsection*{$Spin^c(p,q)$-groups and spinor representations}

%real structures when appropriate later

For the algebraic background we follow \cite{ba81,fr,lm}. Consider $\R^{p,q}$, that is, $\R^n$, where $n=p+q$, equipped with the scalar product $\langle \cdot, \cdot \rangle_{p,q}$ of index $p$, given by $\langle e_i, e_j \rangle_{p,q} = \epsilon_i \delta_{ij}$, where $(e_1,...,e_n)$ denotes the standard basis of $\R^n$ and $\epsilon_{i \leq p} = -1$, $\epsilon_{i > p} = 1$  Let $e_i^{\flat}:=\langle e_i, \cdot \rangle_{p,q} \in \left(\R^{p,q}\right)^*$. We denote by $Cl_{p,q}$ the Clifford algebra of $(\R^{n},- \langle \cdot, \cdot \rangle_{p,q})$ and by $Cl_{p,q}^{\C}$ its complexification. It is the associative real or complex algebra with unit multiplicatively generated by $(e_1,...,e_n)$ with the relations $e_ie_j+e_je_i=-2 \langle e_i,e_j \rangle_{p,q}$.
The representation theory of $Cl(p,q)$ and  ${Cl}^{\C}_{p,q}$ is well-known from \cite{lm,har}.
%It is well-known (cf. \cite{lm},\cite{har}) that if $p-q \not \equiv 1 $mod $4$, there is (up to equivalence) exactly one irreducible real representation of $Cl_{p,q}$. If $p-q \equiv 1 $\text{mod} $4$, there are precisely two inequivalent real irreducible representations of $Cl_{p,q}$. Furthermore, ${Cl}^{\C}_{p,q}$ admits up to equivalence exactly one irreducible complex representation in case $n$ is even and two such representations if $n$ is odd. In case that there are two equivalence classes of irreducible real or complex representations, they can be distinguished by the unit volume element as presented in \cite{lm}: Let . If $p-q \equiv 1$ mod $4$, each irreducible real representation of $Cl_{p,q}$ or $Cl_{p,q}^{\C}$ maps $\omega_{\R}$ to $Id$ or $-Id$. Both possibilities can occur and the resulting representations are inequivalent. The analogous statements are true in the complex case for $Cl_{p,q}^{\C}$ and $n$ odd (cf. \cite{ba81}).
In particular, using the action of the volume element $\omega_{\R}:= e_1 \cdot....\cdot e_{n}  \in Cl_{p,q}$ resp. $\omega_{\C}:= (-i)^{\left[\frac{n+1}{2}\right]-p} \omega_{\R} \in Cl^{\C}_{p,q}$, cf. \cite{lm}, there is a way to distinguish a up to equivalence unique real resp. complex irreducible representation for all Clifford algebras $Cl_{p,q}$ and $Cl_{p,q}^{\C}$. 

\begin{remark} \label{ds}
Let $E,T,g_1$ and $g_2$ denote the $2 \times 2$ matrices
\begin{align*} 
E = \begin{pmatrix} 1 & 0 \\ 0 & 1 \end{pmatrix} \text{ , } T = \begin{pmatrix} 0 & -i \\ i & 0 \end{pmatrix} \text{ , } U = \begin{pmatrix} i & 0 \\ 0 & -i \end{pmatrix} \text{ , } V = \begin{pmatrix} 0 & i \\ i & 0 \end{pmatrix}.
\end{align*}
Furthermore, let $\tau_j =\begin{cases}  1 &  \epsilon_j = 1, \\ i & \epsilon_j = -1. \end{cases}$.
Let $n=2m$. In this case, ${Cl}^{\C}(p,q) \cong M_{2^m}(\C)$ as complex algebras, and an explicit realisation of this isomorphism is given by
\begin{align*}
\Phi_{p,q} (e_{2j-1})&= \tau_{2j-1} E \otimes...\otimes E \otimes U \otimes T \otimes...\otimes T,\\
\Phi_{p,q} (e_{2j})  &= \tau_{2j} E \otimes...\otimes E \otimes V \otimes \underbrace{T \otimes...\otimes T}_{(j-1) \times}.
\end{align*}
For $n=2m+1$ and $q>0$ an isomorphism $\widetilde{\Phi}_{p,q} : {Cl}^{\C}({p,q}) \rightarrow M_{2^m}(\C) \oplus M_{2^m}(\C)$ is given by
\begin{align*}
\widetilde{\Phi}_{p,q} (e_j) &= (\Phi_{p,q-1}(e_j),\Phi_{p,q-1}(e_j)) \text{,  } j=1,...,2m ,\\
\widetilde{\Phi}_{p,q} (e_{2m+1}) & = \tau_{2m+1} (iT \otimes...\otimes T, -iT \otimes...\otimes T),
\end{align*}
and $\Phi_{p,q}:=pr_1 \circ \widetilde{\Phi}_{p,q}$ is an irreducible representation mapping $\omega_{\C}$ to $Id$.
\end{remark}

The Clifford group contains $Spin^+(p,q)$, the identity component of the spin group, as well as the unit circle $S^1 \subset \C$ as subgroups. Together they generate the group $Spin^c(p,q)$ and since $S^1 \cap Spin^+(p,q) = \{ \pm1 \}$, we have $Spin^c(p,q) = Spin^+(p,q) \cdot S^1 = Spin^+(p,q)\times_{\mathbb{Z}_2} S^1$.
$Spin^c(p,q)$ has various algebraic relations to other groups, see \cite{fr}: Let $\lambda:Spin^+(p,q) \rightarrow SO^+(p,q)$ denote the two-fold covering of the special orthogonal group. There are natural maps
% We define (cf...)
\begin{align*}
\lambda^c:Spin^c(p,q) &\rightarrow SO^+(p,q),\text{ } [g,z]\mapsto\lambda(g), \\
\zeta:Spin^c(p,q) &\rightarrow SO^+(p,q) \times S^1,\text{ }[g,z]\mapsto (\lambda(g),z^2),
%l:Spin^c(p,q) &\rightarrow S^1 & l([g,z])&=z^2,\\
%i:Spin(p,q) &\rightarrow Spin^c(p,q) & i(g)&=[g,1],\\
%j:S^1 &\rightarrow Spin^c(p,q) &j(z)&=[1,z]
\end{align*}
%This yields the following commutative diagram with exact rows and columns\\
where $\zeta$ is a 2-fold covering. The Lie algebras of $Spin^+(p,q)$ and $Spin^c(p,q)$ are given by $\mathfrak{spin}(p,q)= \{ e_i \cdot e_j \mid 1 \leq i < j \leq n \}$ and $\mathfrak{spin}^c(p,q) = \mathfrak{spin}(p,q) \oplus i \R$. $\zeta_*$ turns out to be a Lie algebra isomorphism, given by $\zeta_*(e_i \cdot e_j, it) = (2E_{ij},2it)$, where $E_{ij} = -\epsilon_j D_{ij} + \epsilon_{i} D_{ji}$ for the standard basis $D_{ij}$ of $\mathfrak{gl}(n,\R)$. Finally, for $(p,q)=(2p',2q')$, the group $Spin^c(p,q)$ is related to the group $U(p',q')$ of pseudo-unitary matrices as follows: Let $\iota : \mathfrak{gl}(m,\C) \hookrightarrow \mathfrak{gl}(2m,\R)$ denote the natural inclusion and define $F:U(p',q') \rightarrow SO(p,q) \times S^1$ by $f(A) = (\iota A, \text{det }A)$. Then there is exactly one group homomorphism $l:U(p',q') \rightarrow Spin^c(p,q)$ such that
\begin{align*} \zeta \circ l = F \end{align*}

%diagramm

For $n=2m$ or $n=2m+1$, fixing an irreducible complex representation $\rho:{Cl}^{\C}_{p,q} \rightarrow \text{End}\left(\De^{\C} \right)$ on the space of spinors $\De:=\De^{\C} = \C^{2^m}$, for instance $\rho = \Phi$ from Remark \ref{ds}, and restricting it to $Spin^{c}(p,q) \subset Cl^{\C}_{p,q}$ yields the complex spinor representation
\[ \rho: Spin^c(p,q) \rightarrow \text{End}\left(\De^{\C} \right), \text{ }\rho([g,z])(v)=z \cdot \rho(g)(v)=:z \cdot g \cdot v. \]
In case $n$ odd, this yields an irreducible representation of $Spin^c(p,q)$, whereas in case $n=2m$ even $\Delta_{p,q}^{\C}$ splits into the sum of two inequivalent $Spin^c(p,q)$ representations $\De^{\C,\pm}$ according to the $\pm 1$ eigenspaces of $\omega$ (cf. \cite{har,ba81}). In our realisation from Remark \ref{ds} one can find these half spinor modules as follows (cf. \cite{ba81}):
Let us denote by $u(\delta) \in \C^2$ the vector $u(\delta)= \frac{1}{\sqrt{2}} \begin{pmatrix}1 \\ -\delta i \end{pmatrix}, \delta = \pm 1,$ and set $u(\delta_1,...,\delta_m):=u(\delta_1) \otimes...\otimes u(\delta_m)$ for $\delta_{j} = \pm1$. Then $\Delta^{\C,\pm}_{p,q} = \text{span} \{ u(\delta_1,...,\delta_m) \mid \prod_{j = 1}^m \delta_{j} = \pm 1 \}$.\\
As $\R^n \subset Cl_{p,q} \subset Cl^{\C}_{p,q}$, $\rho$ also defines the Clifford multiplication $(X,\ph) \mapsto  X \cdot \ph := \rho(X)(\ph)$ of a vector by a spinor. Further, due to the canonical vector space isomorphism $Cl(p,q) \cong \Lambda^*_{p,q} := \Lambda^* \left(\R^{p,q} \right)^*$, forms act on the spinor module in a natural way.
We consider the Hermitian inner product $\langle \cdot, \cdot \rangle_{\De^{\C}}$ on the spinor module $\De^{\C} = \C^{2^m}$ given by
\[ \langle u, v \rangle_{\De^{\C}} = d \cdot (e_{1} \cdot...\cdot e_{p} \cdot u, v )_{\C}, \]
where $d$ is some power of $i$ depending on $p,q$ and the concrete realisation of the representation only and $( \cdot, \cdot )_{\C}$ is the standard Hermitian inner product on $\C^{2^m}$. In the realisation from Remark \ref{ds} we take $d=i^{p(p-1)/2}$. If $p,q >0$, $\langle \cdot, \cdot \rangle_{\De^{\C}}$ has neutral signature and it holds that 
\begin{align*} \langle X \cdot u, v \rangle_{\De^{\C}} + (-1)^p \langle u, X \cdot v \rangle_{\De^{\C}} = 0\end{align*} for all $u,v \in \De^{\C}$ and $X \in \R^n$. In particular, $\langle \cdot, \cdot \rangle_{\De^{\C}}$ is invariant under $Spin^c(p,q)$.\\ 
Moreover, bilinears can be constructed out of spinors generalizing the well-known Dirac current from the Lorentzian case. Concretely, we associate to spinors $\chi_{1,2} \in \De$ a series of forms $\alpha_{\chi_1,\chi_2}^k \in \Lambda^k_{p,q}$, $k \in \mathbb{N}$, given by 
\begin{align}
\langle \alpha_{\chi_1,\chi_2}^k,\alpha \rangle_{p,q} := d_{k,p} \left( \langle \alpha \cdot \chi_1, \chi_2 \rangle_{\Delta_{p,q}} \right) \textit{  } \forall \alpha \in \Lambda^k_{p,q}. \label{6}
\end{align}
$d_{k,p} \in \{Re, Im\}$depends on the chosen representation but not on $\chi$, ensures that the so defined form is indeed a real form. We set $\alpha_{\chi}^k:= \alpha_{\chi,\chi}^k$ In more invariant notation these forms arise in even dimension as the image of a pair of spinors under the map
\[\Delta \otimes \Delta \stackrel{\langle \cdot, \cdot \rangle_{\Delta}}{\rightarrow} \text{End} (\Delta) \cong Cl^{\C}(p,q) \cong \left(\Lambda^*_{p,q}\right)^{\C} \rightarrow \Lambda^k(p,q). \]
and the following properties are easily checked:
\begin{Proposition} \label{10}
Let $\chi \in \De$ and $k \in \mathbb{N}$.
\begin{enumerate}
\item $\alpha^p_{\chi}=0 \Leftrightarrow \chi = 0$,
\item  $\alpha_{\chi}^k =  \sum_{1 \leq i_1 < i_2 <...<i_k \leq n} \epsilon_{i_1}...\epsilon_{i_k} d_{k,p}  \left(\langle e_{i_1}\cdot...e_{i_k}\cdot \chi, \chi  \rangle_{\Delta_{p,q}} \right) e^{\flat}_{i_1} \wedge...\wedge e^{\flat}_{i_k}$,
\item Equivariance: $\alpha^k_{z\cdot g \cdot \chi } = \lambda(g) (\alpha^k_{\chi})$ for all $k \in \mathbb{N}$, $z \cdot g \in Spin^c(p,q)$ and $\chi \in \Delta_{p,q}$.
\end{enumerate}
\end{Proposition}
%There is an important relation between the structure of $\alpha_{\chi}^p$ and $\text{ker } \chi$:
%\newline
%To every spinor $\chi \in \Delta_{p,q}^{\C}$ we can associate a -possibly trivial- linear subspace \[ \text{ker }\chi :=\{ X \in \Delta_{p,q} \mid X \cdot \chi = 0 \}.\]
%If ker $\chi$ is of maximal dimension $\text{min} (p,q)$, we call the spinor (partially) pure.
%\begin{Lemma}\label{0}
%Let $\chi \in \Delta_{p,q}\backslash \{0 \}$ and let $k:= \text{dim }\text{ker }\chi( \leq p)$. Then $\alpha_{\chi}^p$ can be written as
%\begin{align}
%\alpha_{\chi}^p = l_1^{\flat} \wedge ... \wedge l_k^{\flat} \wedge \widetilde{\alpha}, \label{maxi}
%\end{align}
%where $l_j \in \R^{p,q}$ for $1 \leq j \leq k$ such that $\text{span }\{l_1,...,l_k\} = \text{ker }\chi$ (in particular, this implies that the $l_j$ are lightlike and mutually orthogonal), $\widetilde{\alpha} \in \Lambda^{p-k}\left(\left(\text{ker }\chi\right)^{\bot}\right)$ and (\ref{maxi}) is \textit{maximal} in the sense that there exists no lightlike vector $l_{k+1}$ being orthogonal to $l_i$ for $1 \leq i \leq k$ such that $\alpha_{\chi}^p = l_1^{\flat} \wedge ...\wedge l_k^{\flat} \wedge l_{k+1}^{\flat} \wedge \widetilde{\widetilde{\alpha}}$.
%Moreover, whenever $\alpha_{\chi}^p$ can be written as in (\ref{maxi}) for mutually orthogonal lightlike vectors $l_1,...,l_k$, it follows that $l_1,...,l_k \in \text{ker }\chi$.
%\end{Lemma}

\subsection*{$Spin^c$-structures and spinor bundles} \label{tes}

%The complex analogue of the well-known notion of pseudo-Riemannian spin structures (see \cite{ba81}) leads to the study of $Spin^c(p,q)$-structures.
Let $(M,g)$ be a space-and time-oriented, connected pseudo-Riemannian manifold of  index $p$ and dimension $n=p+q \geq 3$. By $\Pe^g$ we denote the $SO^+(p,q)$-principal bundle of all space-and time-oriented pseudo-orthonormal frames $s=(s_1,...,s_n)$. A $Spin^c$-structure  of $(M,g)$ is given by the data $\left( \Q^c, \Pe_1, f^c \right)$, where $\Pe_1$ is a $S^1$-principal bundle over $M$, $\Q^c$ is a $Spin^c(p,q)$-principal bundle over $M$ which together with $f^c : \Q^c \rightarrow \Pe^g {\times} \Pe_1$ defines a $\zeta-$reduction of the product $SO^+(p,q) \times S^1$-bundle $\Pe^g {\times} \Pe_1$ to $Spin^c(p,q)$.
Existence and uniqueness of $Spin^c-$structures is discussed elsewhere, see \cite{lm}. We will from now on assume that $(M,g)$ admits a $Spin^c-$structure (which is locally always guaranteed) and assume that this structure is fixed.\\
Given a $Spin^c$-manifold, the associated bundle $S^g:=\mathcal{Q}^c \times_{Spin^c(p,q)} \De^{\C}$ is called the \textit{complex spinor bundle} and its elements, i.e. classes are denoted by ${u,v}$. In case $n$ even , it holds that $S^g=S^{g,+} \oplus S^{g,-}$, as $\De=\De^+ \oplus \De^-$. Sections, i.e. elements of $\Gamma(M,S^{g,(\pm)})$ are called (half-)spinor fields.
The algebraic objects introduced in the last section define fibrewise Clifford multiplication $\mu : \Omega^*(M) \otimes S^g \rightarrow S^g$ and an Hermitian inner product $\langle \cdot , \cdot \rangle_{S^g}$ with analogous fibrewise properties.. Moreover, pointwise applying the construction of spinor bilinears (\ref{6}) leads to series of differential forms $\Gamma(M,S^{g}) \otimes \Gamma(M,S^{g}) \rightarrow \Omega^k(M)$ associated to a pair of spinor fields. Dualizing this for $k=1$, leads to the well-known Dirac current $V_{\ph} \in \mathfrak{X}(M)$, cf. \cite{lei,bl}.

Let $\omega^g\in \Omega^1 \left(\Pe^g,\mathfrak{so}(p,q) \right)$ denote the Levi Civita connection $\nabla^g$ on $(M,g)$, considered as a bundle connection. Moreover, fix a connection $A \in \Omega^1 \left(\Pe_1,i\R \right)$ in the $S^1$ bundle. Together, they form a connection $\omega^g \times A$ on $\Pe^g \times \Pe_1$, which lifts to $\widetilde{\omega^g \times A}:= \zeta_*^{-1} \circ (\omega^g \times A) \circ df^c \in \Omega^1 \left(\Q^c,\mathfrak{spin}^c(p,q) \right)$. The covariant derivative $\nabla^{A}$ on $S^g$ induced by this connection can locally be described as follows: Let $\ph \in \Gamma(S^g)$ be locally given by $\ph_{|U} = \left[\widetilde{s \times e},v \right]$, where $s \in \Gamma(U,\Pe^g), e \in \Gamma(U,\Pe_1)$ and $\widetilde{s \times e}$ is a lifting to $\Gamma(U,\Q^c)$. Denoting $A^e:=e^*A \in \Omega^1(U,i\R)$, we have for $X \in TU$:
\begin{align}
\nabla^{A}_X \ph_{|U} = [\widetilde{s \times e},  X(v) + \frac{1}{2} \sum_{1 \leq k < l \leq n} \underbrace{\epsilon_k \epsilon_l g(\nabla^g_X s_k,s_l)}_{=:\omega_{kl}(X)} e_k \cdot e_l \cdot v + \frac{1}{2} A^e(X) \cdot v ]  \label{lofo}
\end{align}

The inclusion of a $S^1$-connection $A$ in the construction of this covariant derivative gauges the natural $S^1$-action on $S^g$ in the following sense: Let $f=e^{2i\sigma}:M \rightarrow S^1$ be a smooth function. Then we have by (\ref{lofo}) that
\begin{align}
\nabla_X^A (f \cdot \ph) = 2i d\sigma(X) \cdot f \cdot \ph + f \cdot \nabla_X^A \ph = f \cdot \nabla_X^{A+i d\sigma} \ph \label{gauge}
\end{align}
It is moreover known from \cite{fr} that for all $X,Y \in \mathfrak{X}(M)$ and $\ph,\psi \in \Gamma(S^g)$ we have
\begin{align*}
\nabla^A_X (Y \cdot \ph) &= \nabla^g_X Y \cdot \ph + Y \cdot \nabla_X^A \ph, \\
X \langle \ph, \psi \rangle_{S^g} &=\langle \nabla_X^A \ph, \psi \rangle_{S^g} + \langle \ph, \nabla^A_X \psi \rangle_{S^g}.
\end{align*}
Let $dA$ denote the curvature form of $A$, seen as element of $\Omega^2(M,i\R)$. Let $R^A$ denote the curvature tensor of $\nabla^A$ and $R^g : \Lambda^2(TM) \rightarrow \Lambda^2(TM)$ the curvature tensor of $(M,g)$. These quantities are related by
\begin{equation} \label{cur}
\begin{aligned}
R^A(X,Y)\ph &= \frac{1}{2} R^g(X,Y)\cdot \ph + \frac{1}{2} dA \cdot \ph\text{, }\sum_i \epsilon_i s_i \cdot R^A(s_i,X) \ph = \frac{1}{2}Ric(X) \cdot \ph - \frac{1}{2} (X \invneg dA) \cdot \ph 
\end{aligned}
\end{equation}

\begin{remark} \label{drr}
Every spin-manifold is canonically $Spin^c$ with trivial auxiliary bundle. Moreover, if one takes for $A$ the canonically flat connection on $M \times S^1$ in this situation, then $\nabla^A$ corresponds to the connection on $S^g$ induced by the Levi Civita connection, see \cite{mor}.\\
For $(M,g)$ a manifold which admits a $U(p',q') \hookrightarrow SO^+(p,q)$ reduction $(\Pe_U,h:\Pe_U \rightarrow \Pe^g)$ of its frame bundle, the bundles $(\Q^c:=\Pe_U \times_l Spin^c(p,q), \Pe_1:= \Pe_U \times_{\text{det}} S^1)$ together with the map
\[ f^c: \Q^c \rightarrow \Pe^g \times \Pe_1, \text{ }[q,zg]_l \mapsto ([q,\lambda(g)],[q,z^2]) \]
define a $Spin^c(p,q)$ structure on $M$. In this situation, there are natural reduction maps
\begin{align}
\phi_c : \Pe_U &\rightarrow \Q^c\text{, }p \mapsto [p,1]_l \\
\phi_1: \Pe_U &\rightarrow \Pe_1\text{, } p \mapsto [p,1]_{det}.
\end{align}
Moreover, local sections in $\Q^c$ can be obtained as follows: Let $s \in \Gamma(U,\Pe_U)$ be a local section. Then we have that $\phi_c(s) \in \Gamma(U,\Q^c)$ and 
\begin{align}f^c(\phi_c (s)) = s \times e\text{, where }e=\phi_1(s). \label{sec} \end{align}
\end{remark}

\subsection*{Basic properties of charged conformal Killing spinors}

Given a pseudo-Riemannian $Spin^c$-manifold $(M,g)$ together with a connection $A$ on the underlying $S^1$-bundle there are naturally associated differential operators: For the Dirac operator $D^A:= \mu \circ \nabla^A : \Gamma(S^g) \rightarrow \Gamma(S^g)$ the Schroeder-Lichnerowicz formula (cf. \cite{fr}) gives 
\begin{align}
D^{A,2} \ph = \Delta^A \ph + \frac{R}{4} \ph + \frac{1}{2} dA \cdot \ph, \label{slr}
\end{align}
where $\Delta_A \ph = - \sum_i \epsilon_i \left(\nabla^A_{s_i} \nabla^A_{s_i} \ph - \text{div}(s_i) \nabla^A_{s_i} \ph \right)$ and $R$ is the scalar curvature of $(M,g)$. The complementary $Spin^c$-twistor operator $P^A$ is obtained by performing the spinor covariant derivative $\nabla^{A}$ followed by orthogonal projection onto the kernel of Clifford multiplication,
\[ P^A : \Gamma(S^g)  \stackrel{\nabla^{A}}{\rightarrow} \Gamma(T^*M \otimes S^g ) \stackrel{g}{\cong} \Gamma(TM \otimes S^g)  \stackrel{\text{proj}_{\text{ker}\mu}}{\rightarrow} \Gamma(\text{ker} \mu). \]
Spinor fields $\ph \in \text{ker }P^A$ are called $Spin^c$-twistor spinors. A local calculation shows that they are equivalently characterized as solutions of the twistor equation
\[\nabla^{A}_X \ph + \frac{1}{n} X \cdot D^A \ph = 0 \text{    for all } X \in \mathfrak{X}(M). \]
Following the conventions in \cite{CCKS1,CCKS2,CCKS3}, we shall call $Spin^c$-twistor spinors \textit{charged conformal Killing spinors} and abbreviate them by \textit{CCKS}.\\
\newline
In analogy to the $Spin$-case, CCKS are objects of conformal $Spin^c-$geometry: Let $f^c_g: \Q^c_g \rightarrow \Pe_+^g \times \Pe_1$ be a $Spin^c(p,q)$-structure for $(M,g)$ and let $\widetilde{g}=e^{2 \sigma}g$ be a conformally equivalent metric. As in the case of spin structures (cf. \cite{ba81,bfkg}), there exists a canonically induced $Spin^c-$structure $f^c_{\widetilde{g}}: \Q_{\widetilde{g}}^c \rightarrow \Pe_+^{\widetilde{g}} \times \Pe_1$ and a $Spin^c(p,q)$-equivariant map $\widetilde{\phi}_{\sigma} : \mathcal{Q}^c_g \rightarrow \mathcal{Q}^c_{\widetilde{g}}$ such that the diagram %for instance bundles equal
\begin{align*}
\begin{xy}
  \xymatrix{
      \mathcal{Q}^c_g  \ar[r]^{\widetilde{\phi}_{\sigma}} \ar[d]_{f^c_g}    &   \mathcal{Q}^c_{\widetilde{g}} \ar[d]^{f^c_{\widetilde{g}}}  \\
      \Pe_+^g \times \Pe_1 \ar[r]_{\phi_{\sigma}}             &   \Pe_+^{\widetilde{g}} \times \Pe_1   
  }
\end{xy}
\end{align*}
commutes, where $\phi_{\sigma}((s_1,...,s_n),e) = \left(\left(e^{-\sigma}s_1,...,e^{-\sigma}s_n \right),e \right)$. We obtain identifications
\begin{align*}
\begin{array}{llllllll} 
\widetilde{}:S^g & \rightarrow & S^{\widetilde{g}}, & \ph&=&[\widehat{q},v] & \mapsto & [\widetilde{\phi}_{\sigma} (\widehat{q}),v ] = \widetilde{\ph}, \\
\widetilde{}: TM & \rightarrow &TM, & X&=&[q,x] & \mapsto & [\phi_{\sigma} (q), x] = e^{- \sigma} X,
\end{array}
\end{align*}
where the second map is an isometry wrt. $g$ and $\widetilde{g}$. With these identifications, the covariant derivative $\nabla^A$ on the spinor bundle, the Dirac operator and the twistor operator transform as
\begin{align*}
\nabla^{A,\widetilde{g}}_{\widetilde{X}} \widetilde{\varphi} &= e^{-\sigma} \widetilde{\nabla^{A,g}_X \varphi} - \frac{1}{2} e^{-2 \sigma} (X \cdot \text{grad}^g(e^{\sigma}) \cdot \varphi + g(X,\text{grad}^g(e^{\sigma})) \cdot \varphi) \widetilde{},\\
D^{A,\widetilde{g}} \widetilde{\varphi} &= e^{-\frac{n+1}{2}\sigma} \left( D^{A,g}(e^{\frac{n-1}{2}\sigma} \varphi) \right) \widetilde{},\\
P^{A,\widetilde{g}} \widetilde{\varphi} &= e^{-\frac{\sigma}{2}} \left( P^{A,g}(e^{-\frac{\sigma}{2}} \varphi) \right) \widetilde{}. 
\end{align*}
Thus, $ P^{A,g}$ is conformally covariant and $\ph \in \text{ker }P^{A,g}$ iff $e^{\sigma/2}\widetilde{\ph} \in \text{ker }P^{A,\widetilde{g}}$. The $S^1$-bundle data, and in particular $A$,  are unaffected by the conformal change. However, applying (\ref{gauge}) directly yields the following additional $S^1-$gauge invariance of the CCKS-equation:

\begin{Proposition} \label{huuk}
Let $\ph \in \text{ker }P^{A,g}$ and $f=e^{i \tau /2} \in C^{\infty}(M,S^1)$. Then $f\ph \in \text{ker }P^{A-i d\tau,g}$ and $D^{A-i d \tau} (f \ph) = f D^A \ph$. 
\end{Proposition}
Consequently, the data needed to define CCKSs are in fact a conformal manifold $(M,[g])$, where we require that $(M,g)$ is $Spin^c$ for one - and hence for all - $g \in c$, and a gauge equivalence class of $S^1$-connections in the underlying $S^1$-bundle $\Pe_1$.

\begin{Proposition} \label{tg}
The following hold for $\ph \in \text{ker }P^{A,g}$:
\begin{align}
D^{A,2} \ph &= \frac{n}{n-1} \left( \frac{R}{4}\ph + \frac{1}{2}dA \cdot \ph \right),\label{1p} \\
\nabla_X^A D^A \ph &= \frac{n}{2} \left(K^g(X) + \frac{1}{n-2} \cdot \left(\frac{1}{n-1} X \cdot dA +X \invneg dA \right) \right) \cdot \ph. \label{2p}
\end{align} 
Here, $K^g:= \frac{1}{n-2} \cdot \left(Ric - \frac{R}{2(n-1)}g \right)$ denotes the Schouten tensor.
\end{Proposition}

\textit{Proof. }All calculations are carried out at a fixed point $x \in M$. Let $(s_1,...,s_n)$ be a pseudo-orthonormal frame which is parallel in $x$. We have at $x$:
\begin{align*}
- \Delta^A \ph + \frac{1}{n} D^{A,2} \ph = \sum_i \epsilon_i \nabla_{s_i}^A \left(\nabla^A_{s_i} \ph + \frac{1}{n} s_i \cdot D^A \ph \right) = 0,
\end{align*}
and thus by (\ref{slr}) $\frac{1}{n}D^{A,2} \ph = \Delta^A \ph = D^{A,2} \ph - \frac{R}{4}\ph - \frac{1}{2} dA \cdot \ph$, from which (\ref{1p}) follows. To prove (\ref{2p}), note that the twistor equation yields $R^A(X,s_i)\ph = -\frac{1}{n} \left( s_i \nabla_X^A D^A\ph - X \cdot \nabla_{s_i}^A D^A \ph \right)$, for $X$ a vector field which is parallel in $x$. Inserting this into (\ref{cur}) implies that
\begin{align*}
Ric(X) \cdot \ph &= \frac{2}{n} (2-n) \nabla_X^A D^{A} \ph + \frac{2}{n} X \cdot D^{A,2} \ph + (X \invneg dA) \cdot \ph \\
&= \frac{2}{n} (2-n) \nabla_X^A D^{A} \ph + \frac{R}{2(n-1)} X \cdot \ph + \frac{1}{n-1}X \cdot dA \cdot \ph + (X \invneg dA) \cdot \ph.
\end{align*}
Solving for $\nabla_X^A D^A \ph$ yields the claim.
$\hfill \Box$\\
\newline
Proposition \ref{tg} leads to an equivalent characterization of CCKS. To this end, consider the bundle $E^g := S^g\oplus S^g$ together with the covariant derivative
\begin{align}
\nabla^{E^g,A}_X \begin{pmatrix}\ph \\ \psi \end{pmatrix}:= \begin{pmatrix} \nabla_X^A \ph + \frac{1}{n} X \cdot \psi \\ \nabla_X^A \psi -\frac{n}{2} \left(K^g(X) + \frac{1}{n-2} \cdot \left(\frac{1}{n-1} X \cdot dA +X \invneg dA \right) \right) \cdot \ph \end{pmatrix}. \label{dat}
\end{align}
Obviously, $\ph \in \text{ker }P^A$ implies that $\nabla^{E^g,A} \begin{pmatrix}\ph \\ D^A \ph \end{pmatrix} = 0$, and on the other hand, if $\nabla^{E^g,A} \begin{pmatrix}\ph \\ \psi \end{pmatrix} = 0$, then $\ph \in \text{ker }P^A$ and $\psi = D^A \ph$. It follows as in the $Spin$-case that for a nontrivial CCKS the spinors $\ph$ and $D^A \ph$ never vanish at the same point and dim ker $P^A\leq 2^{\left\lfloor n/2 \right\rfloor+1}$.

\begin{remark}
The appearance of the $dA$-term in (\ref{dat}) hinders a conformally invariant interpretation of CCKS in terms of parallel spin tractors associated to a parabolic Cartan geometry as known for $Spin-$twistor spinors from \cite{leihabil,baju}.
\end{remark}

\section{Integrability conditions and spinor bilinears} \label{ikon}
We obtain integrability conditions for the existence of CCKS by computing the curvature operator $R^{\nabla^{E^g,A}}$ which has to vanish when applied to $(\ph, D^A \ph)^T$, where $\ph \in \text{ker }P^{A,g}$. Let $pr_{1,2}$ denote the projections onto the corresponding summands of $E^g$. We calculate:

\begin{align*}
pr_1 \left(R^{\nabla^{E^g,A}}(X,Y) \begin{pmatrix} \ph \\ \psi \end{pmatrix} \right) =& \frac{1}{2} \left( R^g(X,Y) -X \cdot K^g(Y) + Y \cdot K^g(X) \right)\cdot  \ph + \frac{1}{2} dA(X,Y) \cdot \ph \\
&- \frac{1}{2(n-2)} \left( \frac{1}{n-1} (X \cdot Y-Y \cdot X)\cdot dA + (X \cdot (Y \invneg dA) - Y \cdot (X \invneg dA))\right) \cdot \ph \end{align*}
With the definition of the Weyl tensor $W^g$ and using the identities
\begin{equation} \label{clid}
\begin{aligned}
X \cdot \omega &= X^{\flat} \wedge \omega - X \invneg \omega, \\
\omega \cdot X &= (-1)^{k} \left(X^{\flat} \wedge \omega + X \invneg \omega \right),
\end{aligned}
\end{equation}
where $X$ is a vector and $\omega$ a $k-$form, we obtain the integrability condition
\begin{equation} \label{int1}
\begin{aligned}
0=& \frac{1}{2} \cdot W^g(X,Y) \cdot \ph + \left(\frac{n-3}{2(n-1)} \cdot dA(X,Y) -\frac{1}{(n-2)(n-1)} \cdot X^{\flat} \wedge Y^{\flat} \wedge dA \right) \cdot \ph \\
&+ \frac{1}{n-2} \left(\frac{1}{n-1}-\frac{1}{2} \right)\cdot \left( X^{\flat} \wedge (Y \invneg dA) - Y^{\flat} \wedge (X \invneg dA) \right)  \cdot \ph.
\end{aligned}
\end{equation}
In particular, ker $P^A$ being of maximal possible dimension implies $W^g=0$ and $dA=0$. The integrability condition resulting from $pr_2 \left(R^{\nabla^{E^g,A}}(X,Y) \begin{pmatrix} \ph \\ D^A \ph \end{pmatrix} \right) =0$ is with the same formulas and the Cotton-York tensor $C^g(X,Y)=(\nabla_X^g K^g)(Y)-(\nabla_Y^g K^g)(X)$, straightforwardly calculated to be
\begin{align*}
0 =& \frac{1}{2}W^g(X,Y) \cdot D^A \ph + \frac{n}{2}C(X,Y) \cdot \ph -\frac{n}{2}\frac{1}{(n-2)(n-1)} \left( Y^{\flat} \wedge \nabla_X dA - X^{\flat} \wedge \nabla_Y dA \right) \\
&-\frac{n}{2(n-1)} \left(g(\nabla_X dA,Y) - g(\nabla_Y dA,X) \right) \cdot \ph - (\frac{1}{(n-2)(n-1)} X^{\flat} \wedge Y^{\flat} \wedge dA + \frac{n-3}{2(n-1)} dA(X,Y)\\
&+ \frac{1}{n-2} ((X \invneg dA) \wedge Y^{\flat}- (Y \invneg dA) \wedge X^{\flat} )) \cdot D^A \ph.
\end{align*}

\begin{remark}
For Riemannian 4-manifolds these integrability conditions are given in \cite{cas}. 
\end{remark}
We now clarify the relation of CCKS to conformal Killing forms. For this purpose, we introduce the following set of differential forms for a spinor field $\ph \in \Gamma(S^g)$ and $k \in \mathbb{N}$:

\begin{equation}\label{varfom}
\begin{aligned}
g \left( \alpha_{\ph}^k, \alpha \right) & := d_k \cdot \langle \alpha \cdot \ph, \ph \rangle_{S^g}, & \alpha \in \Omega^k(M), \\
g \left( \alpha_0^{k+1}, \beta \right) & := \frac{2d_k(-1)^{k-1}}{n} h \left( \langle \beta \cdot D^A \ph, \ph \rangle_{S^g} \right), & \beta \in \Omega^{k+1}(M),\\
g \left( \alpha_{\mp}^{k-1}, \gamma \right) & := \frac{2d_k(-1)^{k-1}}{n} h \left( \langle \gamma \cdot D^A \ph, \ph \rangle_{S^g} \right), & \gamma \in \Omega^{k-1}(M), \\
\end{aligned}
\end{equation}
where $h(z):= \frac{1}{2} \left(z + (-1)^{k\left(p+1 + \frac{k-1}{2} \right)} \overline{z} \right)$. $d_k \in U(1)$ are powers of $i$, ensuring that $\alpha_{\ph}^k$ is indeed a real form. Imposing the $Spin^c$-twistor equation yields for $\ph \in \text{ker }P^A$
\begin{align}
\nabla_X^g \alpha_{\ph}^k = X \invneg \alpha_0^{k+1} + X^{\flat} \wedge \alpha_{\mp}^{k-1}, \label{ncf}
\end{align}
i.e. $\alpha_{\ph}^k$ is a conformal Killing form. Such forms have been studied intensively in \cite{sem,nc}. From (\ref{ncf}) we deduce that $(k+1) \alpha_0^{k+1} = d \alpha_{\ph}^k$ and $(n-k+1)\alpha_{\mp}^{k-1} = d^* \alpha_{\ph}^k$. Moreover, in case $k=1$ (\ref{ncf}) is equivalent to say that $V_{\ph} = \left(\alpha_{\ph}^1\right)^{\sharp}$ is a conformal vector field. Note that under a conformal change of the metric with factor $e^{2 \sigma}$, $\alpha_{\ph}^k$ transforms with factor $e^{(k+1)\sigma}$, and thus $V_{\ph}$ depends on the conformal class only.\\
\newline
We now derive further equations for the \textit{Lorentzian} case and $k=1$. Note that in this case we may set $d_1=1$. Let us introduce further forms for $\ph \in \Gamma(S^g)$ by setting
\begin{align*}
g \left( \alpha_{dA}^j, \alpha \right) & := \frac{1}{(n-2)(n-1)} \cdot \text{Re } \langle dA \cdot \ph, \alpha \cdot \ph \rangle_{S^g}, & \alpha \in \Omega^j(M), \\
g \left( \widetilde{\alpha}_0^{2}, \beta \right) & := \frac{2}{n} \text{Im } \langle \beta \cdot D^A \ph, \ph \rangle_{S^g} , & \beta \in \Omega^{2}(M),\\
\widetilde{\alpha}_{\mp} & := \frac{2}{n} \text{Im } \langle  D^A \ph, \ph \rangle_{S^g}. & \\
\end{align*}
Differentiating the various forms and straightforward calculation reveals that the twistor equation and (\ref{2p}) yield the following system of equations:
\begin{equation}\label{noco}
\begin{aligned}
\begin{pmatrix}
\nabla_X^g & - X \invneg & - X^{\flat} \wedge & 0 \\ -K^{g}(X) \wedge  & \nabla_X^g & 0 & X^{\flat} \wedge \\ -K^{g}(X) \invneg & 0 & \nabla_X^g &  -X \invneg \\ 0 & K^g(X) \invneg & -K^{g} (X)\wedge & \nabla_X^g 
\end{pmatrix} \begin{pmatrix} \alpha_{\ph}^1 \\ \alpha^2_0 \\ \alpha_{\mp} \\ \frac{2}{n^2} \alpha_{D^A \ph}^1 \end{pmatrix} = \begin{pmatrix} 0 \\ \frac{1}{n-2} (X \invneg \frac{1}{i}dA)^{\sharp} \invneg \alpha_{\ph}^3 - X^{\flat} \wedge \alpha_{dA}^1 + X \invneg \alpha_{dA}^3 \\ X \invneg \alpha_{dA}^1 \\ \frac{1}{n-1} \left( \frac{1}{n-2} (X \invneg \frac{1}{i}dA)^{\sharp} \invneg \widetilde{\alpha}_0^2 + \widetilde{\alpha}_{\mp} \cdot (X \invneg \frac{1}{i}dA) \right)\end{pmatrix}
\end{aligned}
\end{equation}

\begin{remark} \label{tdf}
Elements in the kernel of the operator on the left hand side define precisely \textit{normal} conformal Killing forms resp. vector fields. For a conformal vector field $V$, $V^{\flat}$ being normal conformal is equivalent to the curvature conditions (see \cite{cov1,cov2,raj}) $V \invneg W^g = 0$ and $V \invneg C^g = 0 \in \mathfrak{X}(M)$.
Due to the $dA-$terms, the associated vector to a CCKS is in general no normal conformal vector field, in contrast to the $Spin$ setting. In general, there is no additional equation for $\alpha_{\ph}^1$ only, except the conformal Killing equation.
\end{remark}

We next study the relation of $V_{\ph}$ with the two main curvature quantities related to a CCKS, namely $W^g$ and $dA$, for the Lorentzian case. First, we show that $V_{\ph}$ preserves $dA$.
\begin{Proposition} \label{ed}
It holds that $V_{\ph} \invneg \left( \frac{1}{i}dA \right) =  \frac{2(1-n)}{n} d \left(Im \langle D^A \ph, \ph \rangle_{S^g} \right)$. In particular, we have that
\begin{align*}
L_{V_{\ph}} \frac{1}{i}dA = 0.
\end{align*}
\end{Proposition}
\textit{Proof. }
Let us write $\omega = \frac{1}{i}dA \in \Omega^2(M)$. We have for $Y \in TM$:
\begin{align*}
\left(V_{\ph} \invneg \omega \right) (Y)=& \omega(V_{\ph},Y) = - \omega(Y,V_{\ph}) = -g \left( (Y \invneg \omega)^{\sharp},V_{\ph} \right) \\
=&- \langle (Y \invneg \omega) \cdot \ph, \ph \rangle_{S^g} \\
{=}& \frac{1}{i} \frac{2(1-n)}{n} \cdot \langle \nabla_Y^A D^A \ph, \ph \rangle_{S^g} +(n-1) \cdot \underbrace{\frac{1}{i} \langle K^g(Y) \cdot \ph, \ph \rangle_{S^g}}_{\in i \R} \\
&+ \frac{1}{n-2} \cdot \underbrace{\langle (Y^{\flat} \wedge \omega) \cdot \ph, \ph \rangle_{S^g}}_{\in i\R} \in \R \text{ (}(\ref{2p})\text{ inserted)} \\
=& 2 \frac{1-n}{n} \cdot \text{Im} \langle \nabla_Y^A D^A \ph, \ph \rangle_{S^g} = 2 \frac{1-n}{n} \cdot \text{Im} ( Y (\langle D^A \ph, \ph \rangle_{S^g}) + \frac{1}{n} \underbrace{\langle Y \cdot D^A \ph, D^A \ph \rangle_{S^g}}_{\in \R} ) \\
=& 2 \frac{1-n}{n} \cdot d \left(\text{Im }\langle D^A \ph, \ph \rangle_{S^g} \right) (Y).
\end{align*}
The second formula follows directly with Cartans relation $L = \invneg \circ d + d \circ \invneg$.
$\hfill \Box$

\begin{remark}
For 4-d. Lorentzian manifolds \cite{CCKS1} gives an alternative proof of Proposition \ref{ed}.
\end{remark}

Next, we investigate how $V_{\ph}$ inserts into the Weyl tensor. We have by definition and using (\ref{clid})for $X,Y,Z \in TM$:
\begin{equation}\label{ufff}
\begin{aligned}
W^g(V_{\ph},X,Y,Z) &= - \langle \ph, W^g(X,Y,Z) \cdot \ph \rangle_{S^g} \\
 & {=} \langle \ph, Z \cdot W^g(X,Y) \cdot \ph \rangle_{S^g} - \langle \ph, (Z^{\flat} \wedge W^g(X,Y)) \cdot \ph \rangle_{S^g} \in \R 
\end{aligned}
\end{equation}
In Lorentzian signature, $\langle \ph, \omega \cdot \ph  \rangle_{S^g} \in i \R$ for $\omega \in \Omega^3(M)$. Inserting the integrability condition (\ref{int1}) and keeping only real terms, we arrive with the aid of (\ref{clid}) at
\begin{align*}
W^g(V_{\ph},X,Y,Z) = c_n \cdot \langle \ph, \left(Z \invneg (X^{\flat} \wedge Y^{\flat} \wedge dA) + \frac{3-n}{2} Z^{\flat} \wedge ( X^{\flat} \wedge (Y \invneg dA) - Y^{\flat} \wedge (X \invneg dA)) \right) \cdot \ph \rangle_{S^g},
\end{align*}
where $c_n=- \frac{2}{(n-2)(n-1)}$. By permuting $X,Y$ and $Z$, it is pure linear algebra to conclude that the last expression vanishes for all $X,Y,Z \in TM$ if and only if $\langle (X^{\flat} \wedge Y^{\flat} \wedge (Z \invneg dA)) \cdot \ph, \ph \rangle_{S^g} = 0$ for all $X,Y,Z \in TM$. We can express this as follows:

\begin{Proposition} \label{poll}
For a Lorentzian CCKS $\ph \in \text{ker }P^A$, it holds the curvature relation
\begin{align*}
V_{\ph} \invneg W^g = 0 \Leftrightarrow (Z \invneg \frac{1}{i}dA)^{\sharp} \invneg \alpha_{\ph}^3 = 0 \text{ }\forall Z \in TM. 
\end{align*}
\end{Proposition}
In particular, one does not need to compute $W^g$ to check the first condition for $V_{\ph}$ being normal conformal. One obtains another relation between $dA$ and $V_{\ph}$ by requiring the imaginary part of (\ref{ufff}) to vanish. Again, inserting (\ref{int1}) and straightforward manipulations yield with the notation $g((X \invneg dA)^{\sharp},Y):=i \cdot g((X \invneg \frac{1}{i}dA)^{\sharp},Y) \in i\R$ for $X,Y \in TM$:
\begin{align*}
0 =& (3-n) (g(V_{\ph},Z) dA(X,Y) + g(V_{\ph},X) dA(Y,Z) + g(V_{\ph},Y)dA(Z,X) - g(X,Z) g((Y \invneg dA)^{\sharp},V_{\ph})\\
 &+ g(Y,Z) g((X \invneg dA)^{\sharp},V_{\ph}) ) + \frac{2}{n-2} g\left(\alpha_{\ph}^5, X^{\flat} \wedge Y^{\flat} \wedge Z^{\flat} \wedge dA \right)+ i(n-1) \cdot g \left( \alpha_{\ph}^3, Z^{\flat} \wedge W^g(X,Y) \right)
\end{align*}
As a consistency check, note that all integrability conditions including the Weyl curvature become trivial in case $n=3$. Finally, inserting (\ref{int1}) into $g \left(\alpha_{\ph}^2, W^g(X,Y) \right)=i \cdot \langle \ph, W^g(X,Y) \cdot \ph \rangle_{S^g} \in \R$ and splitting into real and imaginary part, we arrive at the relations
\begin{align*}
i \cdot (1-n) g \left(\alpha_{\ph}^2, W^g(X,Y) \right) &= (3-n) dA(X,Y) \langle \ph, \ph \rangle_{S^g} + \frac{2}{n-2} g \left(\alpha_{\ph}^4, X^{\flat} \wedge Y^
{\flat} \wedge dA \right), \\
0 &= \langle \ph, \left(X^{\flat} \wedge (Y \invneg dA) - Y^{\flat} \wedge (X \invneg dA) \right) \cdot \ph \rangle_{S^g}.
\end{align*}

We conclude these general observations about CCKS with some remarks regarding the zero set $Z_{\ph} \subset M$ of a CCKS $\ph \in \text{ker }P^A$. By (\ref{2p}) every $x \in Z_{\ph}$ satisfies $\nabla D^A \ph (x) = 0$. This observation allows one to prove literally as in \cite{bfkg} and \cite{lei} the following:

\begin{Proposition}
Let $\ph \in \text{ker }P^A$ be a CCKS on $(M^{p,q},g)$. If $\gamma:I \rightarrow Z_{\ph} \subset M$ is a curve which runs in the zero set, then $\gamma$ is isotropic. If $p=0$, then $Z_{\ph}$ consists of a countable union of isolated points. If $p=1$, then the image of every geodesic $\gamma_v$ starting in $x \in Z_{\ph}$ with initial velocity $v$ satisfying that $v \cdot D^g \ph(x) = 0$ is contained in $Z_{\ph}$
\end{Proposition}
%This ends our discussion of general properties of the CCKS-equation and its relations to curvature. We next turn to construction principles, classification results and relations to special geometries in small dimensions.

\section{CCKS and CR-Geometry} \label{crg}
\subsection*{The Fefferman metric}
This section gives a global construction principle of CCKS with nontrivial curvature $dA \in \Omega^2(M,i\R)$ on Lorentzian manifolds $(M^{1,2n+1},g)$ starting from $(2n+1)-$dimensional strictly pseudoconvex structures. It can be viewed as the $Spin^c$-analogue of \cite{bafe}. In view of the subsequent considerations, let us review the following well-known fact:\\
\newline
Consider a pseudo-Riemannian K\"ahler manifold $(M^{p,q},g,J)$, where $(p,q)=(2p',2q')$, $p+q=2n$, endowed with its canonical $Spin^c-$structure (cf. Remark \ref{drr}), where the $U(p',q')-$reduction $\Pe_U$ of $\Pe^g$ is given by considering only pseudo-orthonormal bases of the form $(s_1,J(s_1),...,s_n,J(s_n))$. As $J$ is parallel, $\nabla^g$ reduces to a connection $\omega_U^g \in \Omega^1 (\Pe_U,\mathfrak{u}(p',q'))$. By Remark \ref{drr}, $\Pe_U$ and the $S^1$-bundle $\Pe_1$ are related by det-reduction, 
\[ \phi_1: \Pe_U \rightarrow \Pe_1 = \Pe_U \times_{\text{det}} S^1. \]
Whence, there exists a connection $A \in \Omega^1 (\Pe_1,i \R)$, uniquely determined by
\begin{align*}
(\phi_1(s))^*A=:A^{\phi_1 \circ s} = \text{tr} \left(\omega_U^g \right)^s \text{ for }s \in \Gamma(V,\Pe_U).
\end{align*} 
One calculates that $dA(X,Y) = i\cdot Ric^g(X,JY)$.

\begin{Proposition} 
On every pseudo-Riemannian K\"ahler manifold $(M^{p,q},g,J)$ there exists a $\nabla^A$- parallel spinor.
\end{Proposition}
 
\textit{Proof. } As known from \cite{kb} the complex spinor module $\Delta_{2n}^{\C}$ decomposes into $\Delta_{2n}^{\C} = \oplus_{k=0}^{n} \Delta^{k,\C}_{2n}$, where the $\Delta_{2n}^{k,\C}$ are eigenspaces of the action of the K\"ahler form $\Omega = \langle \cdot, J \cdot \rangle_{p,q}$ to the eigenvalue $\mu_k = (n-2k)i$. $\Delta_{2n}^{n,\C}$ turns out to be one-dimensional, in the notation from section \ref{srty} it is spanned by $u(-1,...,-1)$ and acted on trivially by $U(p',q')$, i.e.
\begin{align}
l(U) \cdot u(-1,...,-1) = u(-1,...,-1) \text{ for }U \in U(p',q'). \label{fu}
\end{align}
The global section $\ph \in \Gamma(M,S^g)$, where $\ph_{|V}:= [\phi_c(s),u(-1,...,-1)]$ for $s \in \Gamma(V,\Pe_U)$ with $\phi_c$ given in Remark \ref{drr}, is well-defined by (\ref{fu}), i.e. independent of the chosen $s$. Writing $s^*\omega_{U}^g$ and $(\phi_1(s))^*A$ in terms of $\nabla^g$ is straightforward and then one directly calculates with (\ref{lofo}) that $\nabla^A \ph = 0$.$\hfill \Box$\\

The rest of this section is devoted to the conformal analogue of this construction. We follow \cite{bafe} and sometimes refer to this article when leaving out steps which are identical in our construction. To start with, let us recall some basic facts from (real) CR-geometry, following \cite{bafe,baju,sta}.\\

For $M^{2n+1}$ a smooth, connected, oriented manifold of odd dimension $2n+1$, a real CR-structure is a pair $(H,J)$, where
\begin{enumerate}
\item $H \subset TM$ is a real $2n-$dimensional subbundle,
\item $J:H \rightarrow H$ is an almost complex structure on $H$: $J^2=-Id_{H}$,
\item If $X,Y \in \Gamma(H)$, then $[JX,Y]+[X,JY] \in \Gamma(H)$ and the integrability condition
$N_J(X,Y):=J([JX,Y]+[X,JY])-\left[JX,JY\right]+[X,Y] \equiv 0$ holds.
\end{enumerate}
$(M,H,J)$ is called a (oriented) CR-manifold. We fix a nowhere vanishing 1-form $\theta \in \Omega^1(M)$ with $\theta_{|H} \equiv 0$, which is unique up to multiplication with a nowhere vanishing function and define the Levi-form $L_{\theta}$ on $H$ as
\begin{align*}
L_{\theta}(X,Y):=d\theta(X,JY)
\end{align*}
for $X,Y \in \Gamma(H)$. $(M,H,J,\theta)$ is called a strictly-pseudoconvex pseudo-Hermitian manifold if $L_{\theta}$ is positive definite. In this case, $\theta$ is a contact form and we let $T$ denote the characteristic vector field of the contact form $\theta$, i.e. $\theta(T) \equiv 1$ and $T \invneg d\theta \equiv 0$. \\
Under the above assumptions $g_{\theta}:=  L_\theta + \theta \circ \theta$ defines a Riemannian metric on $M$. Clearly, the $SO^+(2n+1)-$frame bundle $\Pe^{g_{\theta}}_M$ reduces to the $U(n)$ bundle
\begin{align*}
\Pe_{U,H}:=\{(X_1,JX_1,...,X_n,JX_n,T) \mid (X_1,JX_1,...,X_n,JX_n) \text{ pos. oriented ONB of }(H,L_{\theta}) \},
\end{align*}
where $U(n) \hookrightarrow SO^+(2n) \hookrightarrow SO^+(2n+1)$. By Remark \ref{drr} this induces a $Spin^c(2n+1)$-structure $\left(\Q^c_M=\Pe_{U,H} \times_l Spin^c(2n+1),f^c_M\right)$ on $(M,g_{\theta})$, where $Spin^c(2n) \hookrightarrow Spin^c(2n+1)$, with auxiliary bundle $\Pe_{1,M}= \Pe_{U,H} \times_{\text{det}}S^1$ and natural reduction maps 
\[ \phi_{c,M}:\Pe_{U,H} \rightarrow \Q^c_M\text{, } \phi_{1,M}:\Pe_{U,H} \rightarrow \Pe_{1,M}. \]

There is a special covariant derivative on a strictly pseudoconvex manifold, the Tanaka Webster connection $\nabla^W: \Gamma(TM) \rightarrow \Gamma(T^*M \otimes TM)$, uniquely determined by requiring it to be metric wrt. $g_{\theta}$ and the torsion tensor $\text{Tor}^W$ to satisfy for $X,Y \in \Gamma(H)$ that
$\text{Tor}^W(X,Y) = L_{\theta}(JX,Y) \cdot T$ and $\text{Tor}^W(T,X) = -\frac{1}{2} ([T,X]+J[T,JX])$.
Let $\text{Ric}^W \in \Omega^2(M,i\R)$ and $R^W \in C^{\infty}(M,\R)$ denote the Tanaka-Webster Ricci-and scalar curvature (see \cite{bafe}). %define in terms of real data
As $\nabla^W g_{\theta} = 0, \nabla^WT = 0$ and $\nabla^W J = 0$, it follows that $\nabla^W$ descends to a connection $\omega^W \in \Omega^1(\Pe_{U,H},\mathfrak{u}(n))$. In the standard way, this induces a connection $A^W \in \Omega^1(\Pe_{1,M},i\R)$, uniquely determined by
\begin{align*}
(\phi_{1,M}(s))^*A^W = \text{Tr}\left(s^*\omega^W\right),
\end{align*}
where $s \in \Gamma(V,\Pe_{U,H})$ is a local section. Two connections on a $S^1$-bundle over $M$ differ by an element of $\Omega^1(M,i\R)$. Consequently, 
\begin{align*}
A_{\theta}:= A^W + \frac{i}{2(n+1)}R^W \theta
\end{align*}
is a $S^1-$connection on $\Pe_{1,M}$. Let $\pi:\Pe_{1,M} \rightarrow M$ denote the projection. Setting
\begin{align*}
h_{\theta}:=\pi^*L_{\theta} - i\frac{4}{n+2} \pi^*\theta \circ A_{\theta} 
\end{align*}
defines a right-invariant Lorentzian metric on the total space $F:=\Pe_{1,M}$ considered as manifold, the so-called Fefferman metric. Its further properties are discussed in \cite{leihabil,gr}. In particular, one finds that the conformal class $[h_{\theta}]$ does not depend on $\theta$, which is unique up to multiplication with a nowhere vanishing function, but on the CR-data $(M,H,J)$ only.
We next define a natural $Spin^c(1,2n+1)$-structure on the Lorentzian manifold $(F,h_{\theta})$ and show that it admits a CCKS for a natural choice of $A$.

\subsection*{$Spin^c-$characterization of Fefferman spaces}
This subsection is mainly an application of the spinor calculus for $S^1$-bundles with isotropic fibres over strictly pseudoconvex spin manifolds from \cite{bafe} to our case with slight modifications as we are dealing with $Spin^c$-structures. Let $(F,h_{\theta})$ denote the Fefferman space  of $(M,H,J,\theta)$, where $F=\Pe_{1,M} \stackrel{\pi}{\rightarrow} M$ is the $S^1$-bundle. Let $N \in \mathfrak{X}(F)$ denote the fundamental vector field of $F$ defined by $\frac{n+2}{2}i \in i\R$, i.e. $N(f):=\frac{d}{dt}_{|t=0} \left(f \cdot e^{\frac{n+2}{2}it} \right)$ for $f \in F$. For a vector field $X \in \mathfrak{X}(M)$, let $X^* \in \mathfrak{X}(F)$ be its $A_{\theta}-$horizontal lift. We define the $h_{\theta}$-orthogonal timelike and spacelike vectors $s_1:=\frac{1}{\sqrt{2}}(N-T^*)$, $s_2:=\frac{1}{\sqrt{2}}(N+T^*)$ which are of unit length. Let the time orientation of $(F,h_{\theta})$ be given by $s_1$ and the space orientation by vectors $(s_2,X_1^*,JX_1^*,...,X_n^*,JX_n^*)$, where $(X_1,JX_1,...,X_n,JX_n,T) \in \Pe_{U,H}$. Obviously, the bundle
\begin{align*}
\Pe_{U,F}:=\{(s_1,s_2,X_1^*,JX_1^*,...,X_n^*,JX_n^*) \mid (X_1,JX_1,...,X_n,JX_n,T) \in \Pe_{U,H} \} \rightarrow F
\end{align*}
is a $U(n) \hookrightarrow SO^+(1,2n+1)$ reduction of the orthonormal frame bundle $\Pe^{h_{\theta}}_F \rightarrow F$ and $\Pe_{U,F} \cong \pi^* \Pe_{U,H}$. It follows again with Remark \ref{drr} that there is a canonically induced $Spin^c(1,2n+1)$-structure for $(F,h_{\theta})$, namely
\begin{align*}
\left( \mathcal{Q}^c_F:=\Pe_{U,F} \times_l Spin^c(1,2n+1), f^c_F, \Pe_{1,F}:=\Pe_{U,F} \times_{\text{det}}S^1 \right),
\end{align*}
where $U(n) \stackrel{l}{\rightarrow} Spin^c(2n) \hookrightarrow Spin^c(1,2n+1)$, together with reduction maps
\[ \phi_{c,F}:\Pe_{U,F} \rightarrow \Q^c_F\text{, } \phi_{1,F}:\Pe_{U,F} \rightarrow \Pe_{1,F}. \]
There are two distinct natural maps between the $S^1$-bundles $\Pe_{1,F}$ and $F$: Viewing $\Pe_{1,F}$ as the total space of an $S^1-$bundle over the manifold $F$ gives the projection $\pi_F:\Pe_{1,F} \rightarrow F$, whereas the isomorphism $\pi^*\Pe_{U,H} \cong \Pe_{U,F}$ leads to a natural $S^1-$equivariant bundle map
\begin{align*}
\widehat{\pi}_F: \Pe_{1,F} \cong \pi^*\Pe_{U,H} \times_{\text{det}}S^1 &\rightarrow F \cong \Pe_{U,H} \times_{\text{det}}S^1, \\
[v,z] &\mapsto [\pi_U(v),z ],
\end{align*}
with $\pi_U : \pi^*\Pe_{U,H} \rightarrow \Pe_{U,H}$ being the natural projection. The proof of the following statements is a matter of unwinding the definitions:
\begin{Proposition} \label{dum}
Let $s \in \Gamma(V,\Pe_{U,H})$ be a local section for some open set $V \subset M$ and define $\widehat{s} \in \Gamma(\pi^{-1}(V),\Pe_{U,F} \cong \pi^*\Pe_{U,H})$ by $\widehat{s}(f):= (f,s(\pi(f)) \in (\pi^*\Pe_{U,H})_f$. Further, let $\pi_U : \pi^*\Pe_{U,H} \rightarrow \Pe_{U,H}$ be the natural projection. Then the following diagram commutes:
\begin{align*}
\begin{xy}
  \xymatrix{
      F  \ar@{->}[r]^{\widehat{s}} \ar[d]^{\pi}    &   \Pe_{U,F} \ar@{->}[r]^{\phi_{1,F}} \ar[d]^{\pi_U} & \Pe_{1,F} \ar[d]^{\widehat{\pi}_F}  \\
      M \ar@{->}[r]^{s}   &   \Pe_{U,H} \ar@{->}[r]^{\phi_{1,M}}   &  F=\Pe_{1,M} \\
  }
\end{xy}
\end{align*}
\end{Proposition}

\begin{Proposition} \label{groko}
Let $A \in \Omega^1(F,i\R)$ be a connection on the $S^1-$bundle $F=\Pe_{1,M} \stackrel{\pi}{\rightarrow}M$. Then $\widehat{\pi}_F^*A \in \Omega^1(\Pe_{1,F},i\R)$ is a connection on the $S^1-$bundle $\Pe_{1,F} \stackrel{\pi_F}{\rightarrow}F$. Locally, $A$ and $\widehat{\pi}_F^*A$ are related as follows: Let $s \in \Gamma(V,\Pe_{U,H})$ and let $\widehat{s} \in \Gamma (\pi^{-1}(V),\Pe_{U,F})$ be the induced local section as in Proposition \ref{dum}. It holds that
\begin{align*}
\left( \widehat{\pi}_F^* A\right)^{\phi_{1,F}(\widehat{s})} = \pi^*\left(A^{\phi_{1,M}(s)}\right) \in \Omega^1(\pi^{-1}(V),i\R).
\end{align*}
\end{Proposition}
Let us turn to spinor fields on $F$. The $Spin^c(2n+1)$-bundle $\Q^c_M \rightarrow M$ reduces to the $Spin^c(2n)$-bundle $\mathcal{Q}^c_H :=\Pe_{U,H}\times_l Spin^c(2n) \rightarrow M$. We introduce the reduced spinor bundle of $M$,
\[ S_H:=S_H^{g_{\theta}}:=\mathcal{Q}^c_H \times_{\Phi_{2n}} \Delta_{2n}^{\C} \cong \Pe_{U,H}\times_{\Phi_{2n} \circ l} \Delta_{2n}^{\C}. \]
This allows us to express the spinor bundle $S_F:= S_F^{h_{\theta}} \rightarrow F$ as
\begin{equation} \label{idt}
\begin{aligned}
S_F &= \mathcal{Q}^c_F \times_{\Phi_{1,2n+1}} \Delta_{1,2n+1}^{\C} \cong \pi^*\Pe_{U,H} \times_{\Phi_{1,2n+1} \circ l} \Delta^{\C}_{2n+1} \\
& \cong \pi^*S_H \oplus \pi^* S_H.
\end{aligned}
\end{equation}
The second step is purely algebraic and follows from the decomposition of $\Delta_{1,2n+1}^{\C}$ into the sum $\Delta_{2n}^{\C}\oplus \Delta_{2n}^{\C}$ of $Spin^c(2n) \hookrightarrow Spin^c(1,2n+1)$-representations as presented in \cite{bafe}, where $\R^{2n} \hookrightarrow \R^{2n+2}$ via $x \mapsto (0,0,x)$. Under the identification (\ref{idt}) we have (cf. \cite{bafe}, Proposition 18)
\begin{equation} \label{harvest}
\begin{aligned}
s_1 \cdot (\ph, \phi) &= (-\phi, -\ph), \\
s_2 \cdot (\ph, \phi) &= (-\phi, \ph), \\
X^* \cdot (\ph, \phi) &= (-X \cdot \ph, X \cdot \phi).
\end{aligned}
\end{equation}

This identification allows us to define a global section in $\pi^* S_H \oplus (F \times \{0\}) \subset S_F$ in analogy to the K\"ahler case: $u(-1,...,-1) \in \Delta_{2n}^{\C} $ is the (up to $S^1$-action) unique unit-norm spinor in the eigenspace of the K\"ahler form on $\R^{2n}$ to the eigenvalue $-i\cdot n$. For $s:V \subset M \rightarrow P_{U,H}$ a local section we set 
\begin{align*}
\ph(p):= [\phi_{c,F}(\widehat{s}(p)),u(-1,...,-1)], \text{ }p \in \pi^{-1}(V).
\end{align*}
By (\ref{fu}) this is independent of the choice of $s$. Thus, $\ph \in \Gamma(F,S_F)$ is defined. As last ingredient we introduce the connection
\begin{align*}
A:= \widehat{\pi}^*_F A^W + A^W \in \Omega^1(\Pe_{1,F},i\R)
\end{align*}
on the $S^1$-bundle $\Pe_{1,F} \stackrel{\pi_F}{\rightarrow} F$. This is to be read as follows: $\widehat{{\pi}}^*_F A^W$ is a connection on $\Pe_{1,F}$ by Proposition \ref{groko}. Any other connection is obtained by adding an element of $\Omega^1(F,i\R)$, which we choose to be the connection $A^W$ here, i.e. $A= \widehat{\pi}^*_F A^W + \pi_F^*A^W$.

\begin{satz} \label{feff}
The spinor field $\ph \in \Gamma(F,S^{h_{\theta}}_F)$ is a CCKS wrt. $A$, i.e. $\ph \in \text{ker }P^{A,h_{\theta}}$. The curvature $dA \in \Omega^2(F,i\R)$ is given by 
\[ dA = 2 \pi_{F \rightarrow M}^* Ric^W. \]
In particular, $\ph$ descends to a twistor spinor on a spin manifold iff the Tanaka Webster connection is Ricci-flat. The associated vector field $V_{\ph}$ satisfies
\begin{enumerate}
\item $V_{\ph}$ is a regular isotropic Killing vector field,
\item $\nabla_{V_{\ph}}^A \ph = \frac{1}{\sqrt{2}}i \ph$,
\item $V_{\ph}$ is normal, i.e. $V_{\ph} \invneg W^{h_{\theta}} = 0, V_{\ph} \invneg C^{h_{\theta}} = 0$, \\Furthermore, $K^{h_{\theta}}(V_{\ph},V_{\ph})= \text{const}. < 0,$
\item $V_{\ph} \invneg dA = 0.$
\end{enumerate}
\end{satz}
%comm dia
\textit{Proof. }Applying the local formula (\ref{lofo}) to $\ph$ and using Proposition \ref{groko}, we find for a local section $s=(X_1,...,X_{2n},T) \in \Gamma(V,\Pe_{U,H})$ and a vector $Y \in \Gamma(\pi^{-1}(V),TF)$ that
\begin{equation} \label{erty}
\begin{aligned}
\nabla_Y^A \ph_{|\pi^{-1}(V)} =& -\frac{1}{2}h_{\theta}(\nabla_Y^{h_{\theta}}s_1,s_2) s_1 \cdot s_2 \cdot \ph - \frac{1}{2} \sum_{k=1}^{2n} h_{\theta}(\nabla_Y^{h_{\theta}}s_1,X^*_k) s_1 \cdot X_k^* \cdot \ph \\
&+\frac{1}{2} \sum_{k=1}^{2n} h_{\theta}(\nabla_Y^{h_{\theta}}s_2,X^*_k) s_2 \cdot X_k^* \cdot \ph + \frac{1}{2} \sum_{k<l} h_{\theta}(\nabla_Y^{h_{\theta}}X_k^*,X^*_l) X_k^* \cdot X_l^* \cdot \ph\\
&+\frac{1}{2}\left(A^W\right)^{\phi_{1,M}(s)}(d\pi (Y))\cdot \ph +\frac{1}{2}A^W(Y) \cdot \ph,
\end{aligned}
\end{equation}
where for $X \in \mathfrak{X}(M)$, the vector field $X^* \in \mathfrak{X}(F)$ is the horizontal lift wrt. $A_{\theta}$ (not $A^W$!). The calculation of the local connection 1-forms of $\nabla^{h_{\theta}}$ and their pointwise action on the spinor $u(-1,...,-1)$ has been carried out in \cite{bafe}. Inserting these expressions into (\ref{erty}), we arrive at
\begin{align*}
\nabla_{N}^A \ph_{|\pi^{-1}(V)}&=\left(-i \frac{n}{4}\cdot \ph + \frac{1}{2}A^W(N)\cdot \ph, 0 \right), \\
\nabla_{T^*}^A \ph_{|\pi^{-1}(V)}&= \left(i \frac{R^W}{4(n+1)}\ph - \frac{1}{2}\text{Tr } \omega_s(T) + \frac{1}{2}\left(\left(A^W\right)^{\phi_{1,M}(s)}(T)+A^W(T^*)\right) \cdot \ph ,0\right), \\
\nabla_{X^*}^A \ph_{|\pi^{-1}(V)}&= \left(- \frac{1}{2}\text{Tr } \omega_s(T) + \frac{1}{2}\left(\left(A^W\right)^{\phi_{1,M}(s)}(X)+A^W(X^*)\right) \cdot \ph ,0\right) -\frac{1}{4}(X \invneg d\theta)^* \cdot T^* \cdot \ph.\\
\end{align*}
Here, $\omega_s:=s^*A^W \in \Omega^1(V,\mathfrak{u}(n))$. By definition, we have that
\begin{align*}
A^W(N)&=i \cdot \frac{n+2}{2}, \\
\left(A^W\right)^{\phi_{1,M}(s)}(T)+A^W(T^*)&=\text{Tr }\omega_s(T)-i \frac{R^W}{2(n+1)}, \\
\left(A^W\right)^{\phi_{1,M}(s)}(X)+A^W(X^*)&=\text{Tr }\omega_s(X).
\end{align*}
Inserting this and noticing that for $X \in \{X_1,...,X_{2n} \}$ the 1-form $X \invneg  d\theta$ acts on the spinor bundle by Clifford multiplication with $J(X)$, we arrive at
\begin{align*}
\nabla^A_N \ph &= \frac{1}{2}i \ph, \\
\nabla^A_{T^*} \ph &= 0, \\
\nabla^A_{X^*} \ph &= \left(0, -\frac{\sqrt{2}}{4}J(X) \cdot \ph \right).
\end{align*}
Using (\ref{harvest}), we conclude that
\begin{align*}
-s_1 \cdot \nabla_{s_1}^A \ph = s_2 \cdot \nabla_{s_2}^A \ph = X^* \cdot \nabla_{X^*}^A \ph = \left(0, \frac{1}{2 \sqrt{2}} i \ph \right),
\end{align*}
where $X \in \{X_1,...,X_{2n} \}$. This shows that $h(Y,Y) Y \cdot \nabla^A_Y \ph$ is independent of the vector $Y$ with length $\pm 1$, i.e. $\ph \in \text{ker }P^{A}$. It is straightforward to calculate
\begin{align*}
V_{\ph} = - \langle s_1 \cdot \ph, \ph \rangle_{S^{h_{\theta}}} s_1 + \langle s_2 \cdot \ph, \ph \rangle_{S^{h_{\theta}}} s_2 + \sum_{k=1}^{2n} \langle X_k^* \cdot \ph, \ph \rangle_{S^{h_{\theta}}} X_k^* = -s_1 -s_2 = - \sqrt{2} N.
\end{align*}
We conclude that $V_{\ph}$ is regular and isotropic. It is shown in \cite{bafe}, Proposition 19 that the vertical vector field $N$ is Killing. The claimed relation of $V_{\ph}$ to the curvature tensors of $h_{\theta}$ is true for any fundamental vertical vector field on a Fefferman space (see \cite{gr}).\\
It remains to calculate $dA$ and to prove $4.$: Let $s \in \Gamma(V,\Pe_{U,H})$. It holds that (cf. \cite{bafe}) $dA^W = \text{Tr }d\omega_s = Ric^W \in \Omega^2(M,i\R)$. Considered as a 2-form on $F$, the curvature $dA$ is thus using Proposition \ref{groko} given by
\begin{align*}
dA &= d\left(\widehat{\pi}^*_F A^W\right)^{\phi_{1,F}(\widehat{s})} + \pi^*dA^W = \pi^*d \text{Tr }\omega_s + \pi^*Ric^W \\
&= 2 \pi^* Ric^W
\end{align*}
As $dA$ is the lift of a 2-form on $M$, it follows immediately that the fundamental vector field $V_{\ph} = -\sqrt{2}N$ inserts trivially into $dA$.
$\hfill \Box$
\begin{remark}
Generically, we find only one CCKS on the Fefferman space. One can define another natural global section in $S_F$ in analogy to the spin case in \cite{bafe}. However, there is in general no $S^1-$connection which turns it into a CCKS. This is in complete analogy to the K\"ahler case: On a K\"ahler manifold there is a second natural global section in the spinor bundle constructed out of the eigenspinor to the other extremal eigenvalue of the K\"ahler form on spinors which in general is no $Spin^c$-parallel spinor (cf. \cite{mor}).
\end{remark}

As in the $Spin-$case we can also prove a converse of the last statement:
\begin{satz} \label{converse}
Let $(B^{1,2n+1},h)$ be a Lorentzian $Spin^c$-manifold. Let $A \in \Omega^1(\Pe_1,i\R)$ be a connection on the underlying $S^1$-bundle and let $\ph \in \Gamma(S^g)$ be a nontrivial CCKS wrt. $A$ such that
\begin{enumerate}
\item The Dirac current $V:=V_{\ph}$ of $\ph$ is a regular isotropic Killing vector field and $V \cdot \ph = 0$, 
\item $V \invneg W^h=0$ and $V \invneg C^h = 0$, i.e. $V$ is a normal conformal vector field,
\item $V \invneg dA = 0$,
\item $\nabla^A_{V} \ph = i c \ph$, where $c= \text{const} \in \R \backslash \{0 \}$.
\end{enumerate}
Then $(B,h)$ is a $S^1$-bundle over a strictly pseudoconvex manifold $(M^{2n+1},H,J,\theta)$ and $(B,h)$ is locally isometric to the Fefferman space $(F,h_{\theta})$ of $(M,H,J,\theta)$.
\end{satz}

\textit{Proof. }
The proof runs through the same lines as in the $Spin-$case in \cite{bafe} and references given there: 
First, we prove that the Schouten tensor $K:=K^h$ of $(B,h)$ satisfies
\begin{align}
K(V,V) = \text{const.} < 0. \label{ho}
\end{align}
To this end, we calculate using (\ref{2p})
\begin{align*}
V \cdot \nabla_V^A D^A \ph = \frac{n}{2}V \cdot K(V) + c_1 \cdot V \cdot (V \invneg dA) \cdot \ph + c_2 \cdot V \cdot (V^{\flat} \wedge dA) \cdot \ph,
\end{align*}
where the real constants $c_{1,2}$ are specified by (\ref{2p}). However, as $V$ is lightlike and $V \invneg dA = 0$, the last two summands vanish by (\ref{clid}). Consequently, \[
V \cdot \nabla_V^A D^A \ph = \frac{n}{2}V \cdot K^g(V) \stackrel{1.}{=} -n \cdot K(V,V) \cdot \ph. \] 
On the other hand, the twistor equation and our assumptions yield  
\[V \cdot \nabla_V^A D^A \stackrel{1.}{=} \nabla_V^A (V \cdot D^A) = -n \cdot \nabla_V^A \nabla_V^A \ph \stackrel{4.}{=} nc^2 \cdot \ph. \] 
Consequently, $K(V,V) = -c^2$ and (\ref{ho}) holds.\\ 
Regularity of $V$ implies that there is a natural $S^1$-action on $B$,
\begin{align*}
B \times S^1 \ni (p,e^{it}) \mapsto \gamma^V_{t \cdot \frac{L}{2\pi}}(p) \in B,
\end{align*}
where $\gamma_t^V(p)$ is the integral curve of $V$ through $p$ and $L$ is the period of the integral curves. Thus, $M:=B/S^1$ is a $2n+1$-dimensional manifold and $V$ is the fundamental vector field defined by the element $\frac{2\pi}{L}i \in i\R$ in the $S^1$-principal bundle $(B,\pi,M;S^1)$.\\
As $V$ is by assumption normal and satisfies (\ref{ho}), Sparlings characterization of Fefferman spaces applies (see \cite{gr}), yielding that there is a strictly pseudoconvex pseudo-Hermitian structure $(H,J,\theta)$ on $M$ such that $(B,h)$ is locally isometric to the Fefferman space $(F,h_{\theta})$ of $(M,H,J,\theta)$. For more details regarding the construction of the local isometries $\phi_U : B_{|U} \rightarrow F_{|U}$ we refer to \cite{bafe,gr}.
$\hfill \Box$

\section{A partial classification result for the Lorentzian case} \label{pcr}
We give a complete description of Lorentzian manifolds admitting CCKS under the additional assumption that $V_{\ph}$ is normal. The proof closely follows the $Spin-$case from \cite{leihabil}. For a 1-form $\alpha \in \Omega^1(M)$ we define the rank of $\alpha$ to be ${rk}(\alpha):= \text{max} \{ n \in \mathbb{N}\mid \alpha \wedge (d\alpha)^n \neq 0 \}.$

\begin{satz} \label{clar}
Let $(M,g)$ be a Lorentzian $Spin^c$-manifold admitting a CCKS $\ph \in \Gamma(S^g)$ wrt. a connection $A \in \Omega^1(\Pe_1,i\R)$. Assume further that $V:=V_{\ph}$ is a normal conformal vector field. Then locally off a singular set exactly one of the following cases occurs:
\begin{enumerate}
\item It holds that $rk(V^{\flat})=0$ and $||V||^2_g = 0$.\\ $\ph$ is locally conformally equivalent to a $Spin^c$-parallel spinor on a Brinkmann space.
\item It holds that $rk(V^{\flat})=0$ and $||V||^2_g < 0$.\\ Locally, $[g]=[-dt^2+h]$, where $h$ is a Riemannian metric admitting a $Spin^c$-parallel spinor. The latter metrics are completely classified, cf. \cite{mor}.
\item $n$ is odd and  $rk(V^{\flat})=(n-1)/2$ is maximal.\\ $(M,g)$ is locally conformally equivalent to a Lorentzian Einstein Sasaki manifold. There exist geometric $Spin$-Killing spinors $\ph_{1,2}$ on $(M,g)$ which might be different from $\ph$, but satisfying $V_{\ph_{1,2}} = V$.
\item $n$ is even and $rk(V^{\flat})=(n-2)/2$ is maximal.\\ In this case, $(M,g)$ s locally conformally equivalent to a Fefferman space.
\item If none of these cases occurs, there exists locally a product metric $g_1 \times g_2 \in [g]$, where $g_1$ is a Lorentzian Einstein Sasaki metric on a space $M_1$ admitting a geometric Killing spinor $\ph_1$ and $g_2$ is a Riemannian Einstein metric on a space $M_2$ such that $M=M_1 \times M_2$ and $V=V_{\ph_1}$
\end{enumerate}
Conversely, given one of the above geometries with a CCKS of the mentioned type, the associated Dirac current $V$ is always normal.
\end{satz}

\begin{proof}
The condition that $V$ is normal is equivalent to say that $\alpha_{\ph}^1$ is a normal conformal Killing 1-form (cf. Remark \ref{tdf}), which means that the RHS in (\ref{noco}) vanishes. Using tractor calculus for conformal geometries (cf. \cite{cs,baju,leihabil}), we conclude that there exists a 2-form $\alpha \in \Lambda^2_{2,n}$ which is fixed by the conformal holonomy representation $Hol(M,c) \subset SO^+(2,n)$. The system of equations (\ref{noco}) allows us to conclude as in \cite{ns} that $\alpha = \alpha_{\chi}^2$ for a spinor $\chi \in \Delta_{2,n}$. 2-forms induced by a spinor in signature $(2,n)$ have been classified in \cite{leihabil} and the geometric meaning of a holonomy-reduction imposed by such a  fixed $\alpha_{\chi}^2$ is well-understood. The following possibilities can occur:\\
\newline
$\alpha=l_1^{\flat} \wedge l_2^{\flat}$ for $l_1,l_2$ mutually orthogonal lightlike vectors. Using nc-Killing form theory as in the proof of Theorem 10 in \cite{leihabil} we conclude that this precisely corresponds to the first case of Theorem \ref{clar} and that there is locally a metric such that $V$ is parallel. \cite{bl} shows that also the spinor itself is $\nabla^A$- parallel in this case.\\
\newline
$\alpha = l^{\flat} \wedge t^{\flat}$, where $l$ is a lightlike vector and $t$ a orthogonal timelike vector. Using \cite{nc} it follows that there is locally a Ricci-flat metric in the conformal class on which $V$ is parallel. By constantly rescaling the metric, we may assume that $||V||^2=-1$. We have to show that the spinor itself is parallel in this situation. To this end, we calculate:
\begin{align*}
0 = Vg(V,V) = V\langle V \cdot \ph, \ph \rangle = -\frac{1}{n} (\langle V^2 \cdot D^A \ph, \ph \rangle + \langle V \cdot \ph, V \cdot D^A \ph \rangle) =-\frac{2}{n}\text{Re}\langle D^A \ph, \ph \rangle.
\end{align*}
We differentiate this function wrt. an arbitrary vector $X$, use $K^g=0$ and (\ref{2p}) to obtain
\begin{align*}
0 = \text{Re} \langle  c_1 (X \invneg dA) \cdot \ph + c_2 (X^{\flat} \wedge dA) \cdot \ph, \ph \rangle -\frac{1}{n} \text{Re} \langle X \cdot D^A \ph, D^A \ph \rangle.
\end{align*}
The first scalar product vanishes as $\langle  (X \invneg dA) \cdot \ph, \ph \rangle \in i\R$ and $\langle (X^{\flat} \wedge dA) \cdot \ph, \ph \rangle=0$ by Proposition \ref{poll}. Thus, $0=V_{D^A\ph}$ from which in the Lorentzian case $D^A\ph = 0$ follows. It is clear that $\ph$ descends to a $Spin^c-$parallel spinor on the Riemannian factor.\\
\newline
$n$ is odd and $\alpha = (\omega_0)_{|V}$, where $V \subset \R^{2,n}$ is a pseudo-Euclidean subspace of signature $(2,n-1)$ and $\omega_0$ denotes the pseudo-K\"ahler form on $V$. In this case $Hol(M,c) \subset SU(1,(n-1)/2)$. As in \cite{leihabil} we conclude that there is locally a Lorentzian Einstein Sasaki metric $g$ (of negative scalar curvature) in the conformal class. Moreover, $V$ is unit timelike  Killing on this metric and belongs to the defining data of the Sasakian structure. It is known from \cite{boh} that there are geometric Killing spinors $\ph_i$ on $(M,g)$ with $V_{\ph_i}=V$. \\
\newline
$n$ is even and $\alpha=\omega_0$ is the pseudo-K\"ahler form on $\R^{2,n}$. This corresponds to conformal holonomy in $SU(1,n/2)$ and as known from \cite{leihabil} this is locally equivalent to having a Fefferman space in the conformal class on which a CCKS exists by the preceeding section.\\
\newline
$\alpha = (\omega_0)_{|W}$, where $W \subset \R^{2,n}$ is a pseudo-Euclidean subspace of even dimension and signature $(2,k)$, where $4 \leq k < n-2$ and $\omega_0$ denotes the pseudo-K\"ahler form on $W$. In this case, the conformal holonomy representation fixes a proper, nondegenerate subspace of dimension $\geq 2$ and is special unitary on the orthogonal complement. As shown in \cite{leihabil} this is exactly the case if locally there is a metric in the conformal class such that $(M,g)=(M_1 \times M_2,g_1 \times g_2)$, where the first factor is Lorentzian Einstein Sasaki. As mentioned before, there exists a geometric $Spin-$Killing spinor inducing $V$ on $M_1$.\\
\newline
Conversely, if one of the geometries from Theorem \ref{clar} together with a $Spin^c-$CCKS of mentioned type as in the Theorem is given, it follows that $V_{\ph}$ is normal conformal: In the first two cases, $\ph$ is parallel, for which $Ric(X) \cdot \ph = 1/2 (X \invneg dA) \cdot \ph$ is known (see \cite{mor}). We thus have that $(X \invneg dA)^{\sharp} \invneg \alpha_{\ph}^3 \in i\R \cap \R$. Proposition \ref{poll} yields that $V_{\ph} \invneg W^g = 0$. An analogous straightforward but tedious equation yields that $V_{\ph} \invneg C^g = 0$. In cases 3 and 5 of Theorem \ref{clar}, $V$ is normal as it is the Dirac current of a $Spin-$Killing spinor. Case 4 was discussed in the previous section and $V$ is normal by Theorem \ref{feff}.
\end{proof}

\begin{remark}
Examples of Lorentzian manifolds admitting a CCKS with non-normal Dirac current are easy to generate. In fact, \cite{CCKS1} shows that on any Lorentzian 4-manifold admitting a null-conformal vector field $V$, there exists -at least locally- a CCKS $\ph$ such that $V_{\ph}=V$. Given a generic null conformal vector field, it will not be normal conformal, and thus the preceeding Theorem does not apply. In fact (see \cite{lei}), if $V_{\ph}$ is null and normal conformal on a Lorentzian 4-manifold $(M,g)$, then $(M,g)$ is pointwise conformally flat or of Petrov type N.
\end{remark}
%how conclude that first factor admits killing spinor ? 
\begin{remark}
The classification for the Riemannian case differs from the $Spin-$case. For instance, every Riemannian 3-manifold with twistor spinor is conformally flat (see \cite{bfkg}), whereas there are recent examples of 3-dimensional non-conformally flat $Spin^c$-manifolds admitting CCKS, see \cite{gross}.
\end{remark}

\section{Low dimensions} \label{lodi}

In the following, all of our considerations will be \textit{local} on some open, simply-connected set $U \subset M$, i.e. we can always assume that there is a uniquely determined $Spin-$structure, the $S^1$-bundle is trivial and under this identification $A$ corresponds to a 1-form $A \in \Omega^1(U,i\R)$.

\subsection*{5-dimensional Lorentzian manifolds with a CCKS} \label{5l}
The spinor representation in signature $(1,4)$ is quaternionic, i.e. $\Delta_{1,4}:=\Delta_{1,4}^{\C}\cong \mathbb{H}^2$. However, we prefer to work with complex quantities. A Clifford representation on $\C^4$ is given by:
\begin{equation}\label{rie}
\begin{aligned}
e_1 &= \begin{pmatrix} & i & & \\ i & & & \\ & & & i \\ & & i & \end{pmatrix}, e_2 = \begin{pmatrix} & -1 & & \\ 1 & & & \\ & & & -1 \\ & & 1 &  \end{pmatrix}, e_3 = \begin{pmatrix} & & -i & \\ & & & i \\ -i & & & \\ & i & & \end{pmatrix}, e_4 = \begin{pmatrix} & & 1 & \\ & & & -1 \\ -1 & & & \\  &1 & & \end{pmatrix}, \\
e_0 &= \begin{pmatrix} 1 & &&\\ & -1 & & \\ & & -1 & \\ &&&1 \end{pmatrix} .
\end{aligned}
\end{equation}
The $Spin^c(1,4)$-invariant scalar product is given by $\langle v,w \rangle_{\Delta_{1,4}} = (e_0 \cdot v, w )_{\C^4}$. According to \cite{br}, the nonzero orbits of the action of $Spin^+(1,4) \cong Sp(1,1)$ on $\Delta_{1,4}$ are just the level sets of $v \mapsto \langle v,v \rangle_{\Delta_{1,4}} \in \R$.
Consider the spinor $u_1 = {\begin{pmatrix} 1 & 0 & 0 & 0 \end{pmatrix}}^T$. It satisfies
\begin{equation}\label{lor5}
\begin{aligned}
\langle u_1, u_1 \rangle_{\Delta_{1,4}}= 1,\text{ } V_{u_1} = e_0,\text{ } \alpha^2_{u_1} = e_1^{\flat} \wedge e_2^{\flat} + e_3^{\flat} \wedge e_4^{\flat},\text{ }\alpha^2_{u_1} \cdot u_1 = 2i \cdot u_1,
\end{aligned}
\end{equation}
whereas for the spinor $u_0 = {\begin{pmatrix} 1 & 1 & 0 & 0 \end{pmatrix}}^T \in \Delta_{1,4}$ we find 
\begin{equation}\label{lor50}
\begin{aligned}
\langle u_0, u_0 \rangle_{\Delta_{1,4}} = 0,\text{ } V_{u_0} = -2(e_0+e_2),\text{ } \alpha^2_{u_0} = 2(e_1^{\flat} \wedge (e_0^{\flat} + e_2^{\flat})),\text{ } \alpha^2_{u_0} \cdot u_0 = 0.
\end{aligned}
\end{equation}
Here, $\langle \alpha^2_u, \alpha \rangle_{1,4} = \frac{1}{i} \langle \alpha \cdot u, u \rangle_{\Delta_{1,4}} \in \R$ for $\alpha \in \Lambda^2_{1,4}$.\\
\newline
Let $(M^{1,4},g)$ be a Lorentzian $Spin^c$-manifold admitting a CCKS $\ph$ wrt. a $S^1-$connection $A$. Locally, around a given point, one has by omitting singular points either that $\langle \ph, \ph \rangle \neq 0$ or $\langle \ph, \ph \rangle \equiv 0$. In the first case let us assume that $\langle \ph, \ph \rangle > 0$. The analysis for CCKS of negative length is completely analogous. Thus, locally there are only two cases to consider:\\
\newline
\textit{Case 1:}\\ 
We may after rescaling the metric assume that $\ph \in \Gamma(S^g)$ is a CCKS with $\langle \ph, \ph \rangle \equiv 1$. Differentiating the length function and inserting the twistor equation yields that
\begin{align}
\text{Re} \langle X \cdot \ph, \eta \rangle \equiv 0, \label{tataff}
\end{align}
where $\eta := -\frac{1}{5} D^A \ph$. Let $s=(s_0,...,s_4)$ be a local orthonormal frame with local lift $\widetilde{s}$ to the spin structure such that locally $\ph = [\widetilde{s},u_1]$, $V_{\ph} = [s,e_0]$, $\alpha_{\ph}^2=[s,\alpha_{u_1}^2]$ and $\eta = \left[\widetilde{s},\begin{pmatrix} a_1+ia_2 \\ a_3+ia_4 \\ a_5+ia_6 \\ a_7+ia_8 \end{pmatrix} \right]$ for functions $a_1,...,a_8: U \rightarrow \R$. However, (\ref{tataff}) is satisfied iff
\begin{align*}
\eta = \left[\widetilde{s},{\begin{pmatrix} ia_2 & 0 & 0 & a_7+ia_8 \end{pmatrix}}^T \right].
\end{align*}
With this preparation, the conformal Killing equation for $\alpha_{\ph}^2$ (cf. the first line of (\ref{noco})) is straightforwardly calculated to be
\begin{align*}
\nabla^g_X \alpha_{\ph}^2 = \text{const.}\cdot \text{Im }\langle \ph, D^g \ph \rangle_{S^g} \cdot X^{\flat} \wedge V_{\ph}^{\flat}.
\end{align*}
In particular, $d\alpha_{\ph}^2 = 0$ and $\nabla^g_{V_{\ph}} \alpha_{\ph}^2 = 0$ in this scale. We now differentiate $\alpha_{\ph}^2 \cdot \ph = 2i \cdot \ph$ wrt. $V_{\ph}$ to obtain $\alpha_{\ph}^2 \cdot \nabla^A_{V_{\ph}} \ph = 2i \nabla^A_{V_{\ph}} \ph$. We multiply this equation by $V_{\ph}$. By (\ref{lor5}) the actions of $\alpha_{\ph}^2$ and $V_{\ph}$ on spinors commute. Furthermore, $V_{\ph} \cdot \nabla^A_{V_{\ph}} \ph = \eta$ by the twistor equation, leading to $\alpha_{\ph}^2 \cdot \eta=2i \eta$, i.e.
\begin{align*}
\alpha_{\ph}^2 \cdot \eta = 2i \left[\widetilde{s}, {\begin{pmatrix} ia_2 & 0 & 0 & -a_7-ia_8 \end{pmatrix}}^T \right] = 2i \eta = 2i \left[\widetilde{s},{\begin{pmatrix} ia_2 & 0 & 0 & a_7+ia_8 \end{pmatrix}}^T \right].
\end{align*}
Consequently, $D^A \ph = -5ia_2 \cdot \ph$, and thus $\nabla^A_X \ph = ia_2 \cdot X \cdot \ph$. However, it is proven in \cite{mor2} that this forces $a_2$ to be constant, i.e. $\ph$ is a $Spin^c$-Killing spinor or a $Spin^c$-parallel spinor. In the second case, $V_{\ph}$ is parallel as well and the metric splits into a product $(\R,-dt^2) \times(N,h)$ where the Riemannian 4-manifold $(N,h)$ admits a parallel $Spin^c$-spinor.  As moreover $\alpha^2_{\ph}$ descends to a parallel 2-form on $(N,h)$ of K\"ahler type, we conclude that $(N,h)$ is K\"ahler. Conversely, every K\"ahler $Spin^c$-manifold endowed with its canonical $Spin^c$-structure admits parallel spinors. If $\ph$ is an imaginary Killing spinor, Re$\langle \ph, D^A \ph \rangle = 0$, thus $V_{\ph}$ is a timelike Killing vector field of unit length satisfying $V_{\ph} \cdot \ph = \ph$. By a constant rescaling of the metric we may moreover assume that the Killing constant is given by $\pm \frac{i}{2}$. Then it is known from \cite{boh}, Thm. 46, that $V_{\ph}$  defines a (not necessarily Einstein) Lorentzian Sasaki structure. Conversely, by \cite{mor} every Lorentzian Sasaki structure endowed with its canonical $Spin^c$-structure admits imaginary $Spin^c$-Killing spinors.\\
\newline
\textit{Case 2:}\\ 
Let us turn to the case in which the CCKS satisfies $\langle \ph, \ph \rangle \equiv 0$. We first remark that in the $Spin$-case, i.e. $A \equiv 0$, this always implies that the spinor is locally conformally equivalent to a parallel spinor off a singular set (see \cite{lei}, Lemma 4.4.6). As we shall see, in the $Spin^c$-case something more interesting happens:
By passing to a dense subset we may assume that $\ph$ and $V_{\ph}$ have no zeroes. We locally rescale the metric such that $V_{\ph}$ becomes Killing,  which is by (\ref{ncf}) equivalent to 
\begin{align}\text{Re} \langle \ph,D^A \ph \rangle = 0 \label{dasda}
\end{align} in this metric $g$. In the chosen metric we also have (see (\ref{lor50})) that $\alpha_{\ph}^2 = r^{\flat} \wedge V_{\ph}^{\flat}$, where $r$ is a spacelike vector field of constant length orthogonal to $V_{\ph}$. Proceeding exactly as in the first case, i.e. locally evaluating the conditions $\text{Re} \langle X \cdot \ph, D^A \ph \rangle \equiv 0$ (resulting from differentiating the length function) and (\ref{dasda}) and inserting this into the definition (\ref{varfom}) for $\alpha_{\mp}^1$ and using the conformal Killing equation for $\alpha_{\ph}^2$ leads to 
\begin{align}
\alpha_{\mp}^1 = {\text{const.}}_1 \cdot d^*\alpha_{\ph}^2 = {\text{const.}}_2 \cdot \text{Im}\langle \ph, D^A \ph \rangle_{S^g} \cdot V_{\ph}^{\flat}, \label{esse}
\end{align} % f replace by Im
for some real nonzero constants. Conversely, a local computation shows that given a conformal Killing form $\alpha = r^{\flat} \wedge l^{\flat}$ such that $\alpha_{\mp} = f \cdot l^{\flat}$ and $r$ is spacelike, orthogonal to $l$ and of constant length, then $l$ has to be a Killing vector field. We summarize:
\begin{Proposition}
Given a CCKS $\ph \in \text{ker }P^A$ without zeroes such that $\langle \ph, \ph \rangle \equiv 0$, the conformal Killing form $\alpha_{\ph}^2$ satisfies $\alpha_{\ph}^2 = r^{\flat} \wedge V_{\ph}^{\flat}$ for a spacelike vector field $r$. There is a local metric $g \in c$ such that $d^*\alpha_{\ph}^2 = \text{const.}\cdot \text{Im}\langle \ph, D^A \ph \rangle_{S^g} \cdot V_{\ph}^{\flat}$. In this scale, $V_{\ph}$ is Killing.
\end{Proposition}
We will now prove that the converse is also true, i.e. we show:

\begin{Proposition}
Given a zero-free conformal Killing 2-form $\alpha = r^{\flat} \wedge l^{\flat} \in \Omega^{2}(M)$ where $r \in \mathfrak{X}(M)$ is spacelike and of unit length, $l \in \mathfrak{X}(M)$ is a orthogonal lightlike vector field such that $d^* \alpha = f \cdot l^{\flat}$, for some function $f$, then there exists locally $A \in \Omega^1(U,i\R)$ and a CCKS $\ph \in \Gamma(U,S^g_{\C})$ wrt. $A$ such that $\alpha_{\ph}^2 = \alpha$ and $f=\text{const.}\cdot \text{Im}\langle \ph, D^A \ph \rangle_{S^g}$.
\end{Proposition}
\textit{Proof. }
There exists  a local orthonormal frame $s=(s_0,...,s_4)$ such that locally $\alpha = [s,\alpha_{u_0}^2]$. Defining $\ph=[\widehat{s},u_0]$, where $\widehat{s}$ is the local lift of $s$ to the spin structure shows that $\alpha_{\ph}^2 = \alpha$ and $\alpha = r^{\flat} \wedge V_{\ph}^{\flat}$. It is a purely  algebraic observation that $\ph$ is the up to local $U(1)$-action unique spinor field with this property, i.e. the surjective map
\begin{align*}
\Delta_{1,4}^{\C} \supset \{ \epsilon \mid \langle \epsilon, \epsilon \rangle_{\Delta} = 0 \} \mapsto  \alpha_{\epsilon}^2 \in \{ \alpha \mid \alpha = r^{\flat} \wedge l^{\flat}, ||r||_{1,4}^2 = 1, ||l||_{1,4}^2 = 0, \langle r,l\rangle_{1,4} = 0 \} \subset \Lambda^2_{1,4}
\end{align*}
is an $S^1$-fibration. Locally, the mentioned properties of $\alpha$ give a linear system of equations for the local connection coefficients $\omega_{ij}$. By the local formula (\ref{lofo}) the property of $\ph$ being a CCKS becomes a linear system of equations for the $\omega_{ij}$ and the $A_i :=A(s_i) \in C^{\infty}(U,i\R)$. A tedious but straightforward computation shows that there is a unique choice of $A$ such that these equations are indeed satisfied. In our chosen gauge one has that
\begin{equation} \label{alo}
\begin{aligned}
A_1 &= -2i \omega_{34}(s_1), A_2 = -2i \omega_{34}(s_2), A_3 = -2i (\omega_{34}(s_3) +\omega_{14}(s_3)),\\ 
A_4 &= -2i (\omega_{34}(s_4) +\omega_{14}(s_4)), A_0 = -2i \omega_{34}(s_0).
\end{aligned}
\end{equation}
In detail, this argument goes as follows: We have to show that if the locally given 2- form $\alpha = \alpha_{\ph}^2=[s,\alpha_{u_0}^2]=s_1^{\flat} \wedge (s_2^{\flat} + s_0^{\flat})$ where $s=(s_0,...,s_4)$ is a local ONB, is a conformal Killing 2-form such that $\alpha_{\mp} = \widetilde{f} \cdot V^{\flat}_{\ph} = \widetilde{f} \cdot [s, e_2^{\flat} + e_0^{\flat}] =\widetilde{f} \cdot (s_2^{\flat} + s_0^{\flat})$ for some function $\widetilde{f}$, then there is a uniquely determined $A \in \Omega^1(U,i\R)$ such that the spinor $\ph=[\widetilde{s},u_0]$ is a CCKS wrt. $A$. By (\ref{esse}) we then necessarily have that $\widetilde{f}$ is a constant multiple of $\text{Im }\langle D^A \ph, \ph \rangle_{S^g}$. \\
To this end, note that by the equivalent characterization of conformal Killing forms in \cite{sem}, the requirement on $\alpha$ is equivalent to
\begin{align}
X \invneg \nabla^g_Y \alpha + Y \invneg \nabla^g_X \alpha = f \cdot \left(X \invneg (Y^{\flat} \wedge V_{\ph}^{\flat}) + Y \invneg (X^{\flat} \wedge V_{\ph}^{\flat}) \right) \text{ }\forall X,Y \in TM, \label{tata}
\end{align}
where $f=\text{const.} \cdot \widetilde{f}$. We let $X,Y$ run over the local ONB $(s_0,s_1,s_2,s_3,s_4)$ and use
\begin{align*}\nabla^g_X (s^{\flat}_i \wedge s^{\flat}_k ) = \sum_j \epsilon_i \omega_{ij}(X) s^{\flat}_j \wedge s^{\flat}_k + \sum_j \epsilon_k \omega_{kj}(X) s^{\flat}_i \wedge s^{\flat}_j \end{align*}
to obtain that (\ref{tata}) is the following system of linear equations in $\omega_{ij}^k:= \epsilon_i \epsilon_j g(\nabla_{s_k} s_i, s_j)$:

%\begin{equation}
{\allowdisplaybreaks\begin{align*} 
&\omega_{20}^1 = f, \omega_{23}^1+\omega_{30}^1=0, \omega_{24}^1+\omega_{40}^1=0,& \\
&\omega_{20}^2 =0, \omega_{12}^1-\omega_{10}^1=0, \omega_{24}^2+\omega_{40}^2 =0, \omega_{23}^2+\omega_{30}^2=\omega_{13}^1,& \\
&\omega_{13}^1 =-\omega_{20}^3, \omega_{23}^3+\omega_{30}^3=0, \omega_{24}^3+\omega_{40}^3=0,& \\
&\omega_{14}^1+\omega_{20}^4=0, \omega_{23}^4+\omega_{30}^4=0,& \\
&\omega_{20}^0=\omega_{10}^1-\omega_{12}^1, \omega_{23}^0+\omega_{30}^0=-\omega_{13}^1, \omega_{24}^0+\omega_{40}^0=-\omega_{14}^1, \omega_{20}^0=0,& \\
&\omega_{13}^2=0, \omega_{14}^2=0, \omega_{12}^2-\omega_{10}^2 = f, &\\
&\omega_{23}^2+\omega_{30}^2=-\omega_{20}^3, \omega_{13}^3=f, \omega_{14}^3=0, \omega_{12}^3-\omega_{10}^3=0,& \\
&\omega_{24}^2=-\omega_{20}^4, \omega_{14}^2=0, \omega_{13}^4=0, \omega_{14}^4=f, \omega_{12}^4-\omega_{10}^4=0,& \\
&\omega_{13}^2=\omega_{13}^0, \omega_{14}^0=\omega_{14}^2, \omega_{12}^0-\omega_{10}^0 = -f,& \\
&\omega_{23}^3+\omega_{30}^3=0, \omega_{13}^3=f, \omega_{14}^3=-\omega_{13}^4, \\
&\omega_{20}^3=\omega_{23}^0+\omega_{30}^0, \omega_{13}^0=\omega_{10}^3-\omega_{12}^3, \omega_{13}^3=f, \omega_{14}^3=0, \omega_{13}^0=0,&\\
&\omega_{24}^4+\omega_{40}^4=0, \omega_{14}^4=f,& \\
&\omega_{24}^0+\omega_{40}^0=\omega_{20}^4, \omega_{12}^4-\omega_{10}^4=-\omega_{14}^0, \omega_{14}^0=0,& \\
&\omega_{12}^0 - \omega_{10}^0 = 0.&
\end{align*}}
%\end{equation}
It follows that $V_{\ph}$ is a Killing vector field as this is equivalent to
$\epsilon_j (\omega_{2j}^i - \omega_{0j}^i) + \epsilon_i (\omega_{2i}^j-\omega_{0i}^j) = 0$.
On the other hand, by the local formula (\ref{lofo}), the twistor equation for $\ph$ is equivalent to
\begin{align}
\epsilon_i e_i \cdot \left(\sum_{k<l}\omega_{kl}^i e_k \cdot e_l \cdot u_0 + \frac{1}{2} A_i \cdot u_0 \right)= \epsilon_j e_j \cdot \left( \sum_{k<l} \omega_{kl}^i e_k \cdot e_l \cdot u_0 + \frac{1}{2} A_j \cdot u_0 \right), \label{jol}
\end{align}
for $0 \leq i < j \leq 4$ and $A_i :=A(s_i) :U \rightarrow i\R$. Inserting the above $\omega-$equations, it is pure linear algebra to check that (\ref{jol}) holds if and only if we set the local functions $A_i$ as given in (\ref{alo}).
$\hfill \Box$\\

We summarize our observations:

\begin{satz} \label{frt}
Let $\ph \in \Gamma(M,S^g)$ be a CCKS wrt. a connection $A$ on a Lorentzian 5-manifold $(M,g)$. Locally and off a singular set the metric can be rescaled such that exactly one of the following cases occurs:
\begin{enumerate}
\item The spinor is of nonzero length and a parallel $Spin^c$-spinor on a metric product $\R \times N$, where $N$ is a Riemannian 4-K\"ahler manifold with parallel spinor.
\item $\ph$ is an imaginary $Spin^c$-Killing spinor of nonzero length, its vector field $V_{\ph}$ is Killing and defines a Sasakian structure.
\item $|\ph|^2 \equiv 0$. The conformal Killing form $\alpha_{\ph}^2=:\alpha$ can be written as $\alpha= r^{\flat} \wedge l^{\flat}$, where $r \in \mathfrak{X}(M)$ is spacelike and $l \in \mathfrak{X}(M)$ is orthogonal to $r$ and lightlike. There is a scale in which $d^* \alpha = f \cdot l^{\flat}$ for some function $f$.
\end{enumerate}
Conversely, for all the geometries listed in 1.-3. there exists (in case 3. only locally) a $Spin^c$-structure, a $S^1-$connection $A$ and a CCKS $\ph \in \text{ker }P^A$. 
\end{satz}

\begin{remark}
It is easy to verify that the correspondence in the third part of this Theorem descends to parallel objects, i.e. on a Lorentzian $Spin^c$-manifold $(M^{1,4},g)$ there exists a $Spin^c$-parallel spinor of zero length if and only if there is a parallel 2-form of type $\alpha = l^{\flat} \wedge r^{\flat}$. This can be understood well from a holonomy-point of view: The $Spin^+(1,4)$-stabilizer of an isotropic spinor in signature $(1,4)$ is by \cite{br} isomorphic to $\R^3$, thus for a $\nabla^A$-parallel spinor of zero length, we have $Hol(\nabla^A) \subset S^1 \cdot \R^3 \cong SO(2) \ltimes \R^3 \subset SO^+(1,4)$ which is precisely the stabilizer of a 2-form $\alpha=r^{\flat} \wedge l^{\flat}$ of causal type as above under the $SO^+(1,4)-$action. That means, $Hol(\nabla^A)$ fixes an isotropic spinor iff $Hol(\nabla^g)$ fixes a 2-form $\alpha=r^{\flat} \wedge l^{\flat}$.
\end{remark}

Finally, the third case from Theorem \ref{frt} can with (\ref{esse}) be specialized and yields the following spinorial characterization of geometries admitting certain Killing forms, i.e. conformal Killing forms $\alpha$ with $d^* \alpha = 0$:
\begin{satz} \label{satz3}
On every Lorentzian 5-manifold admitting a Killing 2-form of type $r^{\flat} \wedge l^{\flat}$ for a spacelike vector field $r$ of unit length and a orthogonal lightlike vector field $l$, there exists (locally) a CCKS with $\langle \ph, D^A \ph \rangle_{S^g}= 0$ and vice versa.
\end{satz}

%EXAMPLES ????

%Let us an application of the third case: Let $(M^{1,3},g)$ be a Lorentzian 4-manifold and let $l \in \mathfrak{X}(M)$ be a lightlike Killing vector field on $M$. Considert the metric product $(N,h):=(\R \times M, +dt^2 + g)$. Obviously, $\alpha:=dt \wedge l^{\flat}$ satisfies the conditions from Theorem \ref{frt} (we even have that $d^* \alpha = 0$
\subsection*{Other signatures}
We investigate the CCKS-equation on manifolds of signature $(0,5), (2,2)$ and $(3,2)$. Together with the last section and the results from \cite{CCKS1,CCKS2,CCKS3} this yields a complete local description of geometries admitting CCKS in all signatures for dimension $\leq 5$. \\
\newline
\textbf{Signature (0,5)}\\
Let us start with the Riemannian 5-case. A Clifford representation of $Cl_{0,5}$ on $\Delta_{0,5}=\C^4$ is given by (\ref{rie}) where one has to replace the $e_0-$matrix by $-i \cdot e_0$ (see \cite{bfkg}). The $Spin^+(0,5) \cong Sp(2)$-invariant scalar product on $\Delta_{0,5}^{\C}$ is just the usual Hermitian product on $\C^4$ and the nonzero orbits of the $Spin^+(0,5)$-action on spinors are given by its level sets. Let us consider the spinor $u:= \begin{pmatrix} 1 & 0 & 0 & 0 \end{pmatrix}$. We have that $V_u = e_0$, $\alpha_u^2$ is the K\"ahler form on $\text{span}\{e_1,...,e_4\}$ and $\alpha_u^2 \cdot u = 2i \cdot u$
Now exactly the same considerations as carried out for spinors of nonzero length in the Lorentzian case in the previous section reveal the following:
\begin{satz} \label{50}
Let $\ph \in \Gamma(S^g)$ be a CCKS of constant length on a 5-dimensional Riemannian $Spin^c$-manifold $(M,g)$. Locally, exactly one of the following cases occurs:
\begin{enumerate}
\item There is a metric split of $(M,g)$ into a line and a 4-dimensional K\"ahler manifold on which $\ph$ is parallel.
\item After a rescaling of the metric, $\ph$ is a $Spin^c$-Killing spinor to Killing number $\pm \frac{1}{2}$. $V_{\ph}$ is a unit-norm Killing vector field which defines a Sasakian structure.
\end{enumerate}
Conversely, these geometries, equipped with their canonical $Spin^c$ structures, admit $Spin^c$-parallel/Killing spinors.
\end{satz}
Consequently, CCKS in signature $(0,5)$ locally equivalently characterize the existence of Sasakian structures or splits into a line and a K\"ahler 4-manifold in the conformal class.\\
\newline
\textbf{Signature (2,2)}\\
$Cl_{2,2} \cong \mathfrak{gl}(4,\R)$, and thus the complex representation of $Cl^{\C}_{2,2}$ on $\Delta_{2,2}^{\C}=\C^4$ arises as a complexification of the real representation
\begin{align*}
e_1=- \begin{pmatrix}  & &  & 1\\  & & 1 &\\  & 1 & &  \\ 1 & & & \end{pmatrix}, e_2=\begin{pmatrix}  & & 1 &\\  & &  &-1\\  1&  & &  \\  &-1 & & \end{pmatrix}, e_3= \begin{pmatrix}  & & -1 &\\  & &  &-1\\  1&  & &  \\  &1 & & \end{pmatrix}, e_4=\begin{pmatrix}  & &  & 1\\  & & -1 &\\  & 1 & &  \\ -1 & & & \end{pmatrix}
\end{align*}
of $Cl_{2,2}$ on $\Delta_{2,2}^{\R}=\R^4$. In this realisation, $Spin^+(2,2) \cong SL(2,\R) \times SL(2,\R)$ and the indefinite scalar product on $\Delta_{2,2}^{\C}$ given by $(e_1 \cdot e_2 \cdot v,w)_{\C^4}$ satisfies $\langle v,v \rangle_{\Delta} \in i\R$. The nonzero orbits of the $Spin^c(2,2)$-action on $\Delta_{2,2}^{\C,\pm}$ are given by the level sets of $\langle \cdot , \cdot  \rangle_{\Delta}$ where half spinors of zero length are precisely the real half spinors $\Delta_{2,2}^{\R,\pm}$, multiplied by elements of $S^1 \subset \C$. These algebraic observations lead to the following local analysis:\\
\newline
Let $(M^{2,2},g)$ be a $Spin^c(2,2)$-manifold admitting a nontrivial CCKS halfspinor $\ph \in \Gamma(S^g_{\C,\pm})$ wrt. the $S^1-$connection $A$. As we are only interested in local considerations, we may (after passing to open neighbourhoods of a given point and omitting a singular set) assume that $||\ph||^2 \equiv 0$ or $||\ph||^2 \neq 0$ everywhere. In the first case, the $Spin^c(2,2)$-orbit structure shows that $\ph$ can be chosen to be a local section of $S^g_{\R,\pm}$ (see also Proposition \ref{huuk}), i.e. there exists locally a pseudo-orthonormal frame $s=(s_1,...s_4)$ with lift $\widetilde{s}$ such that $\ph=[\widetilde{s},u_{0,\pm}]$ for some fixed spinor $u_{0,\pm} \in \Delta_{2,2}^{\R,\pm}$. As $\ph$ is a CCKS, we must have that 
\begin{align} \epsilon_i s_i \cdot \nabla^A_{s_i} \ph = \epsilon_j s_j \cdot \nabla^A_{s_j} \ph \in \Gamma(S^g_{\C,\pm} = S^g_{\R,\pm} \oplus iS^g_{\R,\pm}) \text{ }\forall 1 \leq i,j \leq 4. \label{sonn} \end{align}
Using the local formula (\ref{lofo}) and splitting (\ref{sonn}) into real and imaginary part, we arrive at $\epsilon_i A(s_i) \cdot s_i = \epsilon_j A(s_j) \cdot s_j$ which is possible only if $A \equiv 0$. Consequently, we are dealing with real $Spin^+(2,2)$ twistor half spinors which have been shown to be locally conformally equivalent to parallel spinors in \cite{lis}.\\
If, on the other hand, the spinor norm is nonvanishing, we may rescale the metric such that $||\ph||^2= \pm i$. Differentiating yields that $\text{Im }\langle X \cdot \ph, D^A \ph \rangle_{S^g} \equiv 0$ for $X \in TM$. It is purely algebraic to check that this is possible only if $D^A \ph = 0$. Moreover, $\alpha_{\ph}^2$ is a constant multiple of the pseudo-K\"ahler form, i.e. $\ph $ is a $Spin^c$-parallel half spinor on a K\"ahler manifold of signature $(2,2)$. We summarize:
\begin{satz} \label{22}
Let $\ph \in \Gamma(S^g_{\C,\pm})$ be a CCKS on a $Spin^c$-manifold $(M^{2,2},g)$. Locally, one of the following cases occurs:
\begin{enumerate}
\item $||\ph||^2 = 0$. This implies $A \equiv 0$. The spinor can be locally rescaled to a parallel pure spinor with normal form of the metric given in \cite{kath,br}.
\item There is a scale such that $||\ph||^2 = \text{const}$. In this case, $\ph$ is a parallel $Spin^c-$CCKS on a pseudo-K\"ahler manifold.
\end{enumerate}
In particular, CCKS half spinors of nonzero length equivalently characterize the existence of pseudo-K\"ahler metrics in the conformal class.
\end{satz}
\textbf{Signature (3,2)}\\
A real representation of $Cl_{3,2}$ on $\Delta_{3,2}^{\R}=\R^4$ is given by
\begin{align*}
e_1 = \begin{pmatrix}  &  &  & -1 \\  &  & 1 &  \\  & 1 &  &  \\ -1 &  &  &  \end{pmatrix}, e_2 = \begin{pmatrix} -1 &  &  &  \\  & 1 &  &  \\  &  & -1 &  \\  &  &  & 1 \end{pmatrix}, e_3 = \begin{pmatrix}  & -1 &  &  \\ -1 &  &  &  \\  &  &  & -1 \\  &  & -1 &  \end{pmatrix}, \\
e_4 = \begin{pmatrix}  & 1 &  &  \\ -1 &  &  &  \\  &  &  & -1 \\  &  & 1 &  \end{pmatrix}, e_5 = \begin{pmatrix}  &  &  & 1 \\  &  & -1 &  \\  & 1 &  &  \\ -1 &  &  &  \end{pmatrix}. 
\end{align*} 
The complex representation on $\Delta_{3,2}^{\C} \cong \C^4$ arises by complexification and in this realisation $Spin^+(3,2) \cong Sp(2,\R)$ . The scalar product $\langle \cdot , \cdot \rangle_{\Delta^{\C}_{3,2}}$ is given by $\langle v,w \rangle_{\Delta^{\C}_{3,2}} = v^T J \overline{w}$, where $J= \begin{pmatrix} 0 & -I_2 \\ I_2 & 0 \end{pmatrix}$. Note that $\langle v,v \rangle_{\Delta^{\C}_{3,2}} \in i \R$. Orbit representatives for the action of $Spin^c(3,2)$ on $\Delta^{\C}_{3,2}$ are $u:= \begin{pmatrix} 1 & 0 & 0 & 0 \end{pmatrix}, u_0 := \begin{pmatrix} i & 1 & 0 & 0 \end{pmatrix}$ and $u_b:= \frac{1}{\sqrt2} \begin{pmatrix} 1 & 0 & ib & 0 \end{pmatrix}$, where $b \in \R \backslash \{0 \}$. One calculates that
\begin{equation} \label{blg}
\begin{aligned}
\langle u,u \rangle_{\Delta^{\C}_{3,2}} = 0, V_u = 0, \alpha_u^2=(e^{\flat}_3 - e^{\flat}_4 ) \wedge ( e^{\flat}_1 - e^{\flat}_5), \alpha_u^2 \cdot u = 0, \\
\langle u_0,u_0 \rangle_{\Delta^{\C}_{3,2}} = 0, V_{u_0}= 2 (e^{\flat}_1 - e^{\flat}_5), \alpha_{u_0}^2 = 2 e_4^{\flat} \wedge (e_1^{\flat} - e_5^{\flat}), \alpha_{u_0}^2 \cdot u_0 =0, \\
\langle u_1,u_1 \rangle_{\Delta^{\C}_{3,2}} = -i, V_{u_1}= - e_2^{\flat}, \alpha_{u_1}^2 = (-e_{1}^{\flat} \wedge e_3^{\flat} + e_4^{\flat} \wedge e_5^{\flat}), \alpha_{u_1}^2 \cdot u_1 =-2i \cdot u_1.
\end{aligned}
\end{equation}
Let $(M^{3,2},[g]=c)$ be a $Spin^c$-manifold with CCKS $\ph \in \text{ker }P^A$. In our local analysis, we have two cases to consider:
In the first case, we find a metric $g \in c$ such that $||\ph||^2 = \pm i.$ Using (\ref{blg}) it follows exactly as in the Lorentzian $(1,4)$-case that after constantly rescaling the metric, $\ph$ is either parallel, in which case by (\ref{blg}) the metric splits into a timelike line and a pseudo-K\"ahler manifold,  or a real or imaginary Killing spinor and $V_{\ph}$, which is a timelike unit Killing vector field, defines a pseudo-Sasakian structure.\\
In the second case, we have that $||\ph||^2 \equiv 0$. If $\ph$ is of orbit type $u \in \Delta_{3,2}^{\R}$ on an open set, it follows exactly as in the signature $(2,2)$ case that $A \equiv 0$, i.e. $\ph$ is an ordinary $Spin-$twistor spinor. The local analysis for this case has been carried out in {hs1,lis}. Thus, we are left with the case that $\ph$ is locally of orbit type $u_0$. However, the analysis of this case is completely analogous to the case of Lorentzian $Spin^c$ CCKS of nonzero length and one gets a one-to-one correspondence to certain conformal Killing forms. Carrying out these steps is straightforward and we arrive at
\begin{satz}
Let $(M^{3,2},g)$ be a $Spin^c-$manifold of signature $(3,2)$ and let $\ph \in \Gamma(S^g)$ be a CCKS wrt. a non-flat $S^1-$connection $A$ satisfying $||\ph||^2 \equiv 0$. Then there is a scale in which the conformal Killing form $\alpha_{\ph}^2$ writes as $\alpha_{\ph}^2 = r^{\flat} \wedge V_{\ph}^{\flat}$, where $r$ is a spacelike vector field of constant length, $V_{\ph}$ is orthogonal to $r$ and lightlike Killing and moreover \[d^* \alpha_{\ph}^2 = \text{const.} \cdot \text{Im}\langle D^A \ph, \ph \rangle_{S^g} \cdot V_{\ph}^{\flat}.\] Conversely, if $\alpha = r^{\flat} \wedge l^{\flat}$ is a conformal Killing form such that $r$ is of constant positive length, $l$ is lightlike and orthogonal to $r$ and $d^* \alpha = f \cdot l^{\flat}$ for some function $f$, then there exists a nontrivial $S^1-$connection $A$ and a up to $S^1-$action unique CCKS $\ph$ wrt. $A$ such that $\alpha_{\ph}^2 = \alpha$ and $f=\text{const.} \cdot \text{Im}\langle D^A \ph, \ph \rangle_{S^g}$.
\end{satz}
\textbf{Acknowledgement} The author gladly acknowledges support from the DFG (SFB 647 - Space Time Matter at Humboldt University Berlin) and the DAAD (Deutscher Akademischer Austauschdienst / German Academic Exchange Service). Furthermore, it is a pleasure to thank Jose Figueroa O Farrill for various discussions about mathematical phsycis. 
\small
%\section*{References}
\bibliographystyle{plain}
\bibliography{literatur}

%\textsc{Andree Lischewski\\
%Humboldt-Universit\"at zu Berlin, Institut f\"ur Mathematik\\
%Rudower Chaussee 25, Room 1.310, D12489 Berlin.\\
%E-Mail: }\texttt{lischews@mathematik.hu-berlin.de}
\end{document}